\documentclass[10pt]{amsart}
\usepackage{amsthm,amssymb,amsmath,amscd,epic,eepic}
\usepackage[all]{xy}

 \newcommand     {\comment}[1]   {}
 \newcommand{\mute}[2] {}
\newcommand     {\printname}[1] {}
\newcommand{\labell}[1] {\label{#1}}

\numberwithin{equation}{section}
\newtheorem {Theorem}                   {Theorem}
\newtheorem*{Theorem*}                   {Theorem}
\newtheorem {refTheorem}[equation]      {Theorem}
\newtheorem {Lemma}[equation]           {Lemma}
\newtheorem* {Lemma*}                    {Lemma}

\newtheorem {Corollary} [equation]      {Corollary}
\newtheorem* {Corollary*}                {Corollary}
\newtheorem {Proposition} [equation]    {Proposition}

\theoremstyle{definition}

\newtheorem*{Definition*}{Definition}

\theoremstyle{remark}
\newtheorem{Remark}[equation]{Remark}
\newtheorem*{Remark*}{Remark}

\newtheorem*{Example*}{Example}

\long\def\symbolfootnote[#1]#2{\begingroup%
\def\thefootnote{\fnsymbol{footnote}}\footnote[#1]{#2}\endgroup}

\def    \cC             {{\mathcal C}}

\def    \cE             {{\mathcal E}}

\def \calM{{\mathcal M}}

\def \Hat 	{\widehat}
\def \ssminus	{\smallsetminus}
\def \eps 	{\epsilon}

\def \del	{{\partial}}
\def \delbar	{\overline{\partial}}
\def \zbar	{\overline{z}}
\def \ubar	{\overline{u}}
\def \Tilde	{\widetilde}

\def    \inv    {^{-1}}

\def	\mod	{/\!/}

\def    \R	{{\mathbb R}}

\def    \C	{{\mathbb C}}
\def    \Z       {{\mathbb Z}}
\def    \bbP       {{\mathbb P}}

\def    \half   {{\frac{1}{2}}}

\begin{document}

\title[Non-hamiltonian actions]
{Non-Hamiltonian actions with isolated fixed points}

\author[Susan Tolman]{Susan Tolman}
\address{Department of Mathematics, University of Illinois at Urbana-Champaign,
Urbana, IL 61801}
\email{stolman@math.uiuc.edu}


\thanks{The author is partially supported by National Science Foundation Grant DMS-1206365.
}

\begin{abstract}
We construct  a non-Hamiltonian symplectic circle action 
on a closed,  connected,  six-dimensional symplectic manifold 
with exactly 32 fixed points.
\end{abstract}

\maketitle

\section{Introduction}
\labell{ss:intro}

Let the circle $S^1 \simeq \R/\Z$ act on a (non-empty) closed, connected symplectic manifold $(M,\omega)$, and let $\xi_M$ be
the associated vector field on $M$.
The action is {\bf symplectic} if it
preserves the symplectic form,  equivalently,
if $ \iota_{\xi_M} \omega$ is closed.
The action is {\bf Hamiltonian} if 
there exists
a {\bf moment map}  $\Psi \colon M \to \R$
satisfying $$d \Psi = - \iota_{\xi_M} \omega,$$
equivalently, if $\iota_{\xi_M} \omega$  is exact.
In this case, we can reduce the number of degrees of freedom by passing 
to the {\bf reduced space} 
$M \mod_{\! t} \, S^1  := \Psi\inv(t)/S^1$, which is a 
symplectic orbifold for all regular $t \in \R$.
Moreover, 
the moment map  is a perfect Morse-Bott function whose critical set 
is the fixed set $M^{S^1}\!.\,$ 
Therefore,  the (equivariant) cohomology and (equivariant) Chern
classes of $M$ are largely determined by the fixed set;
for example,
\begin{equation}
\labell{Morse}
\sum_i \dim  H^i(M;\R) = \sum_i \dim H^i(M^{S^1}\! ; \R).
\end{equation}

This leads to the following important question: 
 What conditions force a symplectic action to be Hamiltonian?
By the discussion above, 
if $H^1(M;\R) = 0$ then
every
symplectic action is  Hamiltonian.
In contrast, equation  \eqref{Morse} implies that symplectic circle actions  with no fixed points,
such as the diagonal circle action on the torus $S^1 \times S^1$,  
are never Hamiltonian. 
Frankel made the first significant progress towards answering the question
by proving that 
a K\"ahler circle action on  a closed, connected
K\"ahler manifold is Hamiltonian exactly if it has fixed points \cite{Fr}.
In contrast, McDuff 
constructed 
a non-Hamiltonian symplectic circle action with fixed tori on
a closed, connected  six dimensional symplectic manifold  \cite{Mc}.
Since then, no new examples with fixed points have been constructed, 
but there has been a great  deal of work  proving that symplectic actions with fixed points must be Hamiltonian when 
various additional criteria are satisfied  
\cite{ On, Mc, Gin, LO, TWe, Gia, Go05, Go06, Fe, Ki06, Ro,  PT, Li,Ja, MPR}.
Nevertheless, the following question, 
asked by  McDuff and Salomon in \cite{MS} and often called the ``McDuff conjecture'', is  open:
 Does there exist a non-Hamiltonian symplectic circle action with  {\em isolated} fixed points
on  a closed, connected  symplectic manifold?
The main goal of this paper is to answer that question in the affirmative.  More precisely, we prove the following theorem:

\begin{Theorem}\labell{main}
There exists a non-Hamiltonian symplectic circle action 
with exactly $32$ fixed points
on a closed, connected, six-dimensional symplectic manifold.
\end{Theorem}

Now let $(M,\omega)$ 
fulfill the conclusions of Theorem~\ref{main}.
Given $n \geq 3$ and a Hamiltonian circle action on $\C \bbP^{n-3}$,
the  diagonal action on
the product  $\C \bbP^{n-3} \times M$ 
is  symplectic but not Hamiltonian.  
Hence, Theorem~\ref{main} has the following corollary.

\begin{Corollary}
Given $n \geq 3$, there exists a non-Hamiltonian symplectic circle action with exactly $32(n-2)$
fixed points on a closed, connected, $2n$-dimensional symplectic manifold. 
\end{Corollary}

\begin{Remark}
McDuff proved that a symplectic circle action on 
a closed, connected,  $2n$-dimensional symplectic manifold with $n \leq 2$
is Hamiltonian exactly if it has fixed points \cite{Mc}.
Thus, the example described in Theorem~\ref{main}  has
the lowest possible dimension. 

In contrast, 
there probably exist examples with fewer fixed points.
A non-Hamiltonian symplectic circle action on a closed symplectic
$2n$-dimensional manifold cannot have exactly one fixed point.  
However,  we can't rule out the possibility of such an
action with two fixed points unless $n \neq 3$
 \cite{PT}.
On the other hand, Jang ruled out the possibility of such an action with
three fixed points \cite{Ja}.
Moreover, if $n$ is odd the number of fixed points is even \cite{PT}.
\end{Remark}

To elaborate on the example in Theorem~\ref{main}, we need  some terminology.
Let a circle act symplectically on a $2n$-dimensional 
closed symplectic manifold $(M,\omega)$.
Given   $p \in M^{S^1}$  
there is a unique multiset of integers $\{w_1,\dots,w_n\}$, 
called {\bf (isotropy)  weights},
so that the induced symplectic representation of $S^1$ on $T_p M$ is isomorphic 
to  
the representation on 
$(\C^n, \sqrt{-1}/2 \sum_i dz_i \wedge d \overline{z}_i)$
given by
$\lambda \cdot z = (\lambda^{w_1} z_1,\dots,\lambda^{w_n} z_n)$.
A  map $\Psi \colon M \to S^1$ is a {\bf generalized moment map}
exactly if
$\Psi^*(dt) = - \iota_{\xi_M} \omega$. 
As in the Hamiltonian case, 
the {\bf reduced space}
$M \mod_{\! t} \, S^1  := \Psi\inv(t)/S^1$ is a $2n -2$ dimensional symplectic orbifold for all regular $t \in \R$. 
The {\bf Duistermaat-Heckman} function of $M$
is the  unique continuous function $\varphi \colon \Psi(M) \to \R$
satisfying
$$\varphi(t) = \int_{M \mod_{\! t} \, S^1} \omega_t^{n-1}$$ 
for all regular $t \in \Psi(M)$, where $\omega_t \in \Omega^2(M \mod_{\! t} \, S^1)$
is the reduced symplectic form \cite{DH}.
A {\bf K3 surface} is a closed, connected  complex surface $(X,I)$
with $H^1(X;\R) = \{ 0 \}$ and trivial  canonical bundle. 
A symplectic form $\sigma \in \Omega^2(X)$
is {\bf tamed} if   $\sigma(v, I(v)) > 0$ for all nonzero tangent vectors $v$; 
it is {\bf K\"ahler} if, additionally, 
$\sigma(I(v),I(w)) = \sigma(v,w)$ for all tangent vectors $v$ and $w$.
In this case, we say that  the triple $(X,I,\sigma)$ is a {\bf tame} (respectively, {\bf K\"ahler})
{\bf K3 surface}.
Finally, 
$\Z_2$ acts holomorphically on the  torus $T = \C^2/(\Z^2 + \sqrt{-1}\, \Z^2)$ 
by the involution $[z] \mapsto [-z]$. The quotient $T/\Z_2$ is a {\bf Kummer surface}; it is a complex orbifold with exactly $16$ isolated singular points with isotropy  $\Z/(2)$.
\\

\noindent
{\bf Theorem 1 redux.}
{\em There exists a non-Hamiltonian symplectic circle action 
on a closed, connected, six-dimensional symplectic manifold $(M,\omega)$
with a generalized moment map $\Psi \colon M \to \R/(4\Z) \simeq S^1$.
The level sets $\Psi\inv(\pm 1)$ each contain $16$ fixed points
with weights $\pm \{-2,1,1\}$; additionally,
$\Z/(2)$ fixes $16$ two-spheres with moment image $[-1,1]$.
Otherwise, the action is free.
The Duistermaat-Heckman function is
$$\varphi(t) = \begin{cases} 
\ 4 + 4 t^2 & -1  \leq t \leq 1 \\
-4 + 16 t - 4t^2 & \ 1  \leq t \leq 3. \\
\end{cases}
$$
Finally, the reduced space $M \mod_{\! t} \, S^1$ is symplectomorphic to a 
tame K3 surface for all $t \in (1,3)$ and diffeomorphic to 
the Kummer surface $T/\Z_2$ for all $t \in (-1,1)$. 
}\\

Our proof of Theorem 1 is adapted from \cite{Ko},
where Kotschick
answered another question of McDuff and Salomon \cite{MS} by constructing a free 
(and therefore non-Hamiltonian) symplectic circle action on a six dimensional 
closed symplectic manifold  with contractible orbits; 
see \S\ref{Kotschick}.
In particular, many of the ideas in \S\ref{ss:K3} and \S\ref{ss:free}
are taken directly from Kotschick's paper, although we handle the technical details differently.
(For example,  he considers K\"ahler
K3 surfaces, but we need to allow tame K3 surfaces because the example we construct
in \S\ref{ss:locally free} is tame but not K\"ahler.)

\begin{proof}[Proof of Theorem~\ref{main}]
Clearly,
$ -4 + 16 t - 4 t^2 > 0$ for all $t \in [1,3].$
Therefore, by
Proposition \ref{existsfree2},  
there exists a free circle action
on a symplectic manifold $(M',\omega')$ with proper moment map 
$\Psi' \colon M' \to (1,3)$ so that, for all $t \in (1,3)$, 
the reduced space $M' \mod_{\! t} \, S^1$ is symplectomorphic to a 
tame K3 surface $(X', I_t', \sigma'_t)$; moreover,
\begin{itemize}
\item $(\sigma'_t,\sigma'_t) = -4 + 16t - 4t^2$; and
\item $[\sigma'_t] =  \kappa' - t \eta'$, 
where $\kappa', \eta'$ induce a primitive embedding $\Z^2 \hookrightarrow H^2(X';\Z)$\footnote{
An embedding $(\ell_1,\ell_2) \mapsto \ell_1 \kappa' + \ell_2 \eta'$
is {\bf primitive} if every lattice element that is a {\em real} linear combination of $\kappa'$ and $\eta'$
is also an {\em integral} linear combination.}.
\end{itemize}

Let $(M_+,\omega_+,\Psi_+)$ and $(M_-,\omega_-,\Psi_-)$ be the symplectic manifolds with 
locally free\footnote{A Lie group $G$ acts {\bf locally freely} on $M$
if the set of $g \in G$ with non-empty fixed set $\{x \in M \mid g \cdot x = x\}$ is discrete.}
 circle actions and proper moment maps described in
Proposition~\ref{existslf}. 
(For now, ignore the complex structure.)
The reduced space $M_\pm \mod_{\! t} \, S^1$ is diffeomorphic
to the Kummer surface $T/\Z_2$ for all $t \in \R$.
Moreover, the Duistermaat-Heckman function of $M_\pm$ is $4 + 4t^2$; see Remark~\ref{varphi}.
Finally, 
$\Psi_+\inv(-1,1) \subset M_+$ and  $\Psi_-\inv(-1,1) \subset M_-$
are equivariantly symplectomorphic.

Let $\big(\Tilde{M}_+,\Tilde{\omega}_+,\Tilde{\Psi}_+\big)$ and
$\big(\Tilde{M}_-,\Tilde{\omega}_-,\Tilde{\Psi}_-\big)$ 
be the symplectic manifolds with circle actions 
and proper moment maps described in Proposition~\ref{existsfixed}.
Fix $\epsilon > 0$ sufficiently small.
The preimages $\Tilde{\Psi}_\pm(\pm (-\infty,1+\epsilon))$
each contain exactly $16$ fixed points;
each lies in  $\Tilde\Psi_\pm (\pm 1)$ and has weights 
$\pm \{-2, 1,1\}$.
Additionally, there exist  $a_\pm < b_\pm$ in $\pm (0,1)$ so that
$\Psi_\pm\inv(a_\pm, b_\pm) \subset M_\pm $ 
and $\Tilde{\Psi}_\pm\inv(a_\pm,b_\pm) \subset \Tilde{M}_\pm$  are equivariantly symplectomorphic.
Finally, for all $t \in \pm ( 1, 1 + \epsilon)$, the reduced space 
$\Tilde{M}_\pm \mod_{\!  t} \, S^1$  is symplectomorphic to a tame K3 surface 
$\big(\Tilde{X},\Tilde{I},(\Tilde{\sigma}_\pm)_t\big)$; moreover,
\begin{itemize}
\item $\big((\Tilde{\sigma}_\pm)_t, (\Tilde{\sigma}_{\pm})_t \big)   = -4 \pm 16 t - 4 t^2$; and
\item 
$\big[(\Tilde{\sigma}_\pm)_t\big] = \Tilde{\kappa} -  t \Tilde{\eta}_\pm$,
where $\Tilde{\kappa},\Tilde{\eta}_\pm$ induce a primitive embedding
$\Z^2 \hookrightarrow H^2(\Tilde{X};\Z)$.
\end{itemize}

The first and third paragraphs above imply that
\begin{gather*}
 \big((\Tilde{\sigma}_+)_t, (\Tilde{\sigma}_+)_t \big) =
\big(\sigma'_{t}, \sigma'_{t }\big)
\mbox{ for all } t \in (1, 1 + \epsilon); \quad \mbox{and}  \\
 \big((\Tilde{\sigma}_-)_t, (\Tilde{\sigma}_-)_t \big) =
\big(\sigma'_{t+4}, \sigma'_{t + 4}\big)
\mbox{ for all } t \in (-1 - \epsilon, -1). 
\end{gather*}
Thus, by Proposition~\ref{freeunique2}  they
imply that there exist  $a'_\pm < b'_\pm$ in  $ \pm (1,1+\epsilon)$
so that $\Tilde\Psi_+\inv(a'_+,b'_+) \subset \Tilde{M}_+$ 
and $(\Psi')\inv(a'_+,b'_+) \subset M'$ are equivariantly symplectomorphic, and also
$\Tilde\Psi_-\inv(a'_-,b'_-) \subset \Tilde{M}_-$  and
$(\Psi')\inv(a'_- + 4,b'_- + 4) \subset M'$ are equivariantly symplectomorphic.

Therefore, we can glue together $(\Psi')\inv(a_+',b_-' + 4) \subset M'$,
$\Psi_\pm\inv(a_-,b_+) \subset M_\pm$, $\Tilde{\Psi}_+\inv(a_+,b_+') \subset \Tilde{M}_+$, and $\Tilde{\Psi}_-\inv(a'_-,b_+) \subset \Tilde{M}_-$
to construct a symplectic circle action
on a closed, connected six-dimensional symplectic manifold $M$ with 
a generalized moment map $\Psi \colon M \to \R/(4\Z)$ satisfying all
requirements.
In particular, the action is not Hamiltonian because there is no fixed point
with all positive (or negative) weights.
\end{proof}

The structure of this paper is straightforward.
In Section~\ref{ss:K3}, we classify tame K3 surfaces up to symplectomorphism.
In Section~\ref{ss:free}, 
we use  this classification to  analyze  free Hamiltonian circle actions on symplectic manifolds with reduced spaces symplectomorphic to tame K3 surfaces. 
Under favorable conditions, these are classified by their Duistermaat-Heckman function.
In Section~\ref{ss:locally free}  we construct 
locally free Hamiltonian circle actions on symplectic manifolds 
with reduced spaces diffeomorphic to the Kummer surface $T/\Z_2$.
Finally, in  Section~\ref{ss:fixed},
we use results from \cite{TWa} to add fixed points to the examples constructed in \S\ref{ss:locally free}.
In that  paper, which is joint with Jordan Watts,
we extend some important constructions and theorems  from the 
symplectic and K\"ahler categories to the tame category.

\section{K3 surfaces}
\labell{ss:K3}

In this section, 
we  classify  tame K3 surfaces
up to symplectomorphism.
This is a fairly straightforward consequence of  classification of 
marked K\"ahler K3 surfaces, which we  review,
closely following \cite{BPV}\footnote{
However, since we  classify K3 surfaces up to symplectomorphism, 
our definition of ``K\"ahler K3 surface" includes the K\"ahler form
(see \S\ref{ss:intro}), while
\cite{BPV} only include its cohomology class.
}.

Let $L$ be the {\bf K3 lattice}, that is, the 
even unimodular lattice with signature $(3,19)$; let $L_\R := L \otimes_\Z \R$.
If $(X,I)$ is a K3 surface,
there is an {\bf isometry}  from  $H^2(X;\Z)$ (with the cup product pairing)
to $L$
\cite[Proposition VIII.3.2]{BPV},
that is, an isomorphism of groups that preserves the  symmetric bilinear forms.
Additionally,  any two K3 surfaces are diffeomorphic \cite[Corollary VIII.8.6]{BPV}.
Given  
a vector space $V$  with  symmetric bilinear form $( \cdot, \cdot)$ and $k \in \Z$,
let $G_k^+(V)$ denote the manifold of oriented $k$-planes in $V$
on which
$( \cdot, \cdot)$  is positive definite.
The manifold $G_3^+(L_\R)$ has two components, and so
we can state our classification as follows:

\begin{Proposition}\labell{classifyk3}
Let $X$ be a manifold admitting a K3 structure.
\begin{enumerate}
\item Given $\kappa \in H^2(X;\R)$, there is a tame K3  structure
$(I,\sigma)$ on $X$ with $[\sigma] = \kappa$ if and only if $(\kappa,\kappa) > 0$.
\item Given tame K3 structures $(I_0,\sigma_0)$ and $(I_1,\sigma_1)$ on $X$
and an isometry $\phi \colon H^2(X;\Z) \to H^2(X;\Z)$, there  is
a symplectomorphism $f$ from $(X,\sigma_0)$ to $(X,\sigma_1)$  satisfying
$f^* = \phi$  if  and only if
$\phi\big([\sigma_1]\big) = [\sigma_0]$ and
$\phi$ preserves the components of $G_3^+(H^2(X;\R))$.
\end{enumerate}
\end{Proposition}

To prove this, we begin with a brief review of complex surfaces.
Let $(X,I)$ be a  {\bf complex surface}, that is, a two-dimensional complex manifold;
assume that $X$ is closed and connected.
Given non-negative integers $p$ and $q$ with $p+q = 2$, 
the Dolbeault cohomology $H^{p,q}(X)$
is naturally isomorphic to the subspace of the de Rham cohomology $H^{2}(X;\C)$ whose elements can be represented 
by a $d$-closed form of type $(p,q)$
\cite[Theorem IV.2.9]{BPV}.
Moreover, under this identification, 
$$H^2(X;\C) = H^{2,0}(X) \oplus H^{1,1}(X) \oplus H^{0,2}(X).$$
A class $d$ in the
{\bf Picard lattice}  $H^{1,1}(X) \cap H^2(X;\Z)$\footnote{
By a slight abuse of notation, 
we identify $H^2(X;\Z)$ with its image in $H^2(X;\R)$ and $H^2(X;\R)$
with its image in $H^2(X;\C)$. 
Since the cohomology of K3 surfaces is torsion-free, we trust that
this will not cause confusion.}
is {\bf effective} if there exists an effective divisor $D$
so that $c_1(\mathcal{O}_X(D)) = d$.
A class $\kappa \in H^2(X;\R)$ is {\bf tamed} (respectively, 
{\bf K\"ahler})
if it contains a tamed (respectively,  K\"ahler) form.
We will need the following fact.
\begin{Lemma}\labell{tamedKahler}
Let $(X,I)$ be a  closed connected complex surface. 
Given $\kappa \in H^2(X;\R)$, the  orthogonal projection $\Hat\kappa$
of $\kappa$ onto $H^{1,1}(X;\R)$
is tamed exactly if  $\kappa$ is.
\end{Lemma}

\begin{proof}
By the discussion above, there exists a $d$-closed form
$\alpha \in \Omega^{2,0}(X)$  so that $\kappa = \Hat\kappa + [\alpha + \overline{\alpha}]$.
Moreover,  $\alpha(v,I(v)) = \overline{\alpha}(v,I(v)) = 0$ for all $\alpha \in \Omega^{2,0}(X)$ and all tangent vectors $v$.
\end{proof}

Now assume that $\dim H^1(X;\R)$ is even,  and  
define $$H^{1,1}(X;\R) := H^{1,1}(X) \cap H^2(X;\R).$$
Because the signature of $H^{1,1}(X;\R)$ is $(1, \dim H^{1,1}(X)  - 1)$ \cite[Theorem IV.2.13]{BPV},
the set
\begin{equation}\labell{cones}
\left\{ \left. x \in H^{1,1}(X;\R) \ \right\vert \  (x ,x ) > 0 \right\}
\end{equation}
consists of two disjoint connected cones.

Finally, assume that  $(X,I)$ admits a K\"ahler form.
Since the K\"ahler classes  form a convex subcone of  \eqref{cones},
they lie in one of the two cones;  we call it the {\bf positive cone} and 
denote it by $\cC_X$.

Now let 
$(X',I')$ be another complex surface satisfying the assumptions above.
A {\bf Hodge isometry}   is an isometry $\phi \colon H^2(X';\Z) \to H^2(X;\Z)$\footnote{
Where convenient, we identify $\phi \colon H^2(X';\Z) \to H^2(X;\Z)$ with
its $\R$-linear  and
$\C$-linear extensions.}
that preserves the Hodge decomposition,
that is, 
$$\phi\big(H^{p,q}(X')\big) = H^{p,q}(X) \quad \mbox{for all }
p + q = 2.$$
A  Hodge isometry $\phi$ is {\bf effective} if it preserves the 
positive cones and induces
a bijection between the respective sets of effective classes.

By definition, every K3 surface $(X,I)$ is a  closed, connected complex
surface with $\dim H^1(X;\R) = 0$ even.
Moreover, Siu proved that every K3 surface admits a K\"ahler form \cite[Theorem VIII.14.5]{BPV}.
Therefore, we can restate \cite[Theorem VIII.11.1]{BPV} as follows:

\begin{refTheorem}[Torelli theorem]\labell{torelli}
Let $(X,I)$ and $(X',I')$ be  K3 surfaces.
Given  an effective  Hodge isometry $\phi \colon H^2(X',\Z) \to  H^2(X,\Z)$, 
there exists a biholomorphism
$g \colon X \to X'$ such that $g^* = \phi$.
\end{refTheorem}

This has the following implication for tame K3 surfaces.

\begin{Corollary}\labell{tamedtorelli1}
Let $(X,I,\sigma)$ and $(X',I',\sigma')$ be 
tame K3 surfaces.  Given a Hodge isometry 
 $\phi \colon H^2(X';\Z) \to H^2(X;\Z)$
such that $\phi\big(\left[\sigma'\right]\big) = [\sigma]$,
there exists a biholomorphism
$g \colon X \to X'$ such that $g^* = \phi$.
\end{Corollary}

\begin{proof}
By Lemma~\ref{tamedKahler},  there exists tamed symplectic forms $\Hat\sigma \in \Omega^2(X)$
and $\Hat{\sigma}' \in \Omega^2(X')$ so that $[\Hat\sigma]$ and $[\Hat{\sigma}']$ are the
orthogonal projection of $[\sigma]$ and $[\sigma']$ 
onto $H^{1,1}(X;\R)$ and $H^{1,1}(X';\R)$,
respectively.  
Since  
$\phi\big([\sigma']\big) = [\sigma]$ and
$\phi$ is a Hodge isometry, 
 $\phi\big(\left[\Hat{\sigma}'\right]\big) = [\Hat\sigma]$.
Following \cite[\S VIII.3]{BPV}, we let the {\bf K\"ahler cone} of $X$
be the set of $x$ in the positive cone $\cC_X$ 
such that $(x,d) > 0$ for all effective classes $d$.
Since the tamed classes in $H^{1,1}(X;\R)$  form  a convex subcone of the set \eqref{cones}, the class
$[\Hat\sigma]$ must lie in  $\cC_X$.  Moreover, since $\Hat\sigma$ is tamed,
$(\Hat\sigma,d) > 0$ for all effective classes $d$. 
Therefore, $[\Hat\sigma]$ lies in the K\"ahler cone of $X$.
By a similar argument, $[\Hat{\sigma}']$ lies in the K\"ahler cone of $X'$.
Thus, $\phi$ is effective by Theorem VIII.3.10 in \cite{BPV}.
\end{proof}

Consider the complex  manifold
$$\Omega := \left\{ [\alpha] \in \bbP(L \otimes_\Z \C) 
\mid (\alpha, \alpha) = 0 \mbox{ and } (\alpha, \overline{\alpha}) > 0 \right\}.$$
A {\bf marked K3 surface} is a 
K3 surface $(X,I)$ and  an isometry $\phi \colon H^2(X;\Z) \to L$.
Since the canonical bundle of $X$ is trivial,
$H^{2,0}(X) \simeq \C$. 
Thus, there is a well-defined
{\bf period point} associated to $(X,I,\phi)$: 
\begin{equation}\labell{tau1}
\tau_1(X,I,\phi) = [\phi(\alpha_X)] \in \Omega,
\end{equation}
where $\alpha_X \in H^{2,0}(X)$ 
is any non-zero class. 
Additionally, there is a  {\bf refined period point}
associated to each marked tame K3 surface $(X,I,\sigma,\phi)$:
\begin{equation}\labell{tau2}
\tau_2(X,I,\sigma, \phi) = \big(\phi\big([\sigma]\big), \tau_1(X,I,\phi) \big) \in L_\R \times \Omega,
\end{equation}
where $\tau_1$ is  defined as in  \eqref{tau1}.

Marked tame K3 surfaces $(X,I,\sigma,\phi)$ and $(X',I',\sigma',\phi')$
are {\bf isomorphic} if there exists a 
marked biholomorphism $g \colon X \to X'$ 
with $g^* \big( \left[\sigma'\right] \big) = [\sigma]$.
Here, we say that a diffeomorphism
$g \colon X \to X'$ is {\bf marked} if $g^* = \phi\inv \circ \phi'$.
The composition
$\phi\inv \circ \phi' \colon H^2(X';\R) \to H^2(X;\R)$ is a Hodge isometry that takes $[\sigma']$ to $[\sigma]$
exactly if $\tau_2(X,I,\sigma,\phi) = \tau_2(X',I',\sigma',\phi')$. 
Hence,
Corollary~\ref{tamedtorelli1} can be reformulated to show
that  isomorphism classes of marked tame K3 surfaces  are
classified by the image on the refined period map.

\begin{Corollary}\labell{tamedtorelli}
Two marked tame K3 surfaces
$(X,I,\sigma,\phi)$ and $(X',I',\sigma',\phi')$  
are isomorphic exactly if
$$\tau_2(X,I,\sigma,\phi) = \tau_2(X',I',\sigma',\phi').$$
\end{Corollary}

Our next goal it to describe the image or the refined period map.
Define manifolds
$$\Tilde{K\Omega} 
:= \left\{(\kappa,[\alpha]) \in L_\R \times \Omega  \ \left| \
(\kappa, \kappa) (\alpha, \overline{\alpha}) > 2|(\kappa,\alpha)|^2 \right. \right\} $$
and
$$ K \Omega := \left\{\big(\kappa, [\alpha]\big) \in L_\R \times \Omega 
\left| \  
(\kappa, \kappa) > 0 \mbox{ and } (\kappa, \alpha) = 0 \right. \right\} \subset \Tilde{K\Omega}.$$

Given $\alpha \in L \otimes_\Z \C$, write
 $\alpha = \Re(\alpha) + \sqrt{-1}\, \Im(\alpha) $,
where $\Re(\alpha), \Im(\alpha) \in L_R := L \otimes_\Z \R$.
Then $[\alpha] \in \Omega$ exactly if
\begin{equation*}
(\Re (\alpha), \Im(\alpha)) = 0 \mbox{ and }
(\Re(\alpha) , \Re(\alpha)) = (\Im(\alpha), \Im(\alpha)) = (\alpha, \overline{\alpha})/2> 0.  
\end{equation*}
Hence, given $\kappa \in L_\R$ and $[\alpha] \in \Omega$, the orthogonal projection 
$\Hat\kappa$ of
$\kappa$ onto $\alpha^\perp \cap L_\R$ is 
\begin{equation}\labell{kappahat}
\Hat\kappa = \kappa -
 \frac{(\kappa,\Re (\alpha))}{(\Re (\alpha), \Re (\alpha))} \Re (\alpha)
 - \frac{ (\kappa,\Im (\alpha))}{( \Im (\alpha), \Im (\alpha))} \Im (\alpha),
\end{equation}
and so
\begin{equation}\labell{kappahat2}
(\Hat\kappa, \Hat\kappa) 
= (\kappa,\kappa) -  \frac{  (\kappa,\Re (\alpha))^2}{(\Re (\alpha), \Re (\alpha))} 
 -  \frac{  (\kappa,\Im (\alpha))^2}{(\Im (\alpha), \Im (\alpha))} 
= (\kappa,\kappa) -  \frac{ 2 |(\kappa,\alpha)|^2}{(\alpha,\overline{\alpha})} 
.
\end{equation}
Therefore,
$\kappa$, $\Re (\alpha)$, and $\Im (\alpha)$ are linearly independent
for every
$(\kappa, [\alpha]) \in \Tilde{K \Omega}$;

Define $\Pi \colon \Tilde{K \Omega} \to G_3^+(L_\R)$
by letting  $\Pi(\kappa,[\alpha])$
be the oriented three plane spanned by the  oriented basis 
$\{ \kappa, \Re (\alpha), \Im (\alpha)\}$.
Define spaces
\begin{gather*}
G_3^+(L_\R)^\circ := \left\{ \left. V \in G_3^+(L_\R) \ \right| \  
d  \not\in  V^\perp \mbox{  for  any } d \in L \mbox{ with } (d,d) = - 2 \right\}, \\
(\Tilde{K\Omega})^\circ := \big\{  (\kappa,[\alpha]) \in \Tilde{K\Omega}\
\ \big| \  \Pi(\kappa,[\alpha]) \in G_3^+(L_\R)^\circ \big\},
\end{gather*}
and
$$
(K\Omega)^\circ := \left\{ \left.(\kappa,[\alpha]) \in K\Omega 
\ \right| \  \Pi(\kappa,[\alpha]) \in G_3^+(L_\R)^\circ \right\} \subset (\Tilde{K \Omega})^\circ
.$$

Given $(\kappa, [\alpha]) \in L_\R \times \Omega$, let $\Hat\kappa$ be the 
orthogonal projection of $\kappa$ onto  $\alpha^\perp \cap L_\R$. 
By \eqref{kappahat} and \eqref{kappahat2},
\begin{gather}\labell{compare}
(\kappa,[\alpha]) \in \Tilde{K \Omega}   \Leftrightarrow  
(\Hat\kappa,[\alpha]) \in K \Omega \mbox{  and } 
(\kappa,[\alpha]) \in (\Tilde{K \Omega})^\circ \Leftrightarrow
(\Hat\kappa,[\alpha]) \in (K \Omega)^\circ.
\end{gather}

Our next step is to prove isomorphism classes
of marked tame K3 surfaces are classified by
$(\Tilde{K \Omega})^\circ$. To do this, we show
that the image of the refined period map
is a subset of $(\Tilde{K \Omega})^\circ$ in Lemma~\ref{image}, and 
that it is surjective in Corollary~\ref{existstamek3}.
In fact, we prove a stronger claim; roughly, the period map is
``surjective for paths".
This reflects the fact that $(K\Omega)^\circ$ does not just classify
isomorphism classes of marked K\"ahler K3 surfaces; it is 
the moduli space marked K\"ahler K3 surfaces. (See the proof
of Theorem~\ref{existsk3}.)

\begin{Lemma}\labell{image}
Let $(X,I,\sigma,\phi)$ be a marked tame K3 surface. Then
$\tau_2(X,I,\sigma,\phi) \in (\Tilde{K\Omega})^\circ;$
moreover, if $\sigma$ is K\"ahler then
$\tau_2(X,I,\sigma,\phi) \in ({K\Omega})^\circ.$
\end{Lemma}

\begin{proof}
Note that $H^{1,1}(X;\R) = \alpha_X^\perp \cap H^2(X;\R)$ for any 
non-zero $\alpha_X \in H^{2,0}(X)$. 
Hence, if $\sigma$ is K\"ahler then $[\sigma] \in \alpha_X^\perp$,
and so $\tau_2(X,I,\sigma,\phi)
\in K\Omega$.
Similarly, returning to the general case,
Lemma~\ref{tamedKahler} and \eqref{compare} together imply that
$\tau_2(X,I,\sigma,\phi) \in \Tilde{K \Omega}$.
Now consider $d \in H^2(X;\Z)$ with $(d,d) = -2$.
If $d \in \alpha_X^\perp$, then
since $d \in  H^{1,1}(X) \cap H^2(X;\Z)$
either $d$ or $- d$ is effective
\cite[Proposition VIII.3.6.i]{BPV},
and so $d \not\in [\sigma]^\perp$.
Therefore,   $\tau_2(X,I,\sigma,\phi)$ lies in
$(\Tilde{K\Omega})^\circ$.
\end{proof}

\begin{refTheorem}\labell{existsk3}
Given a smooth path $\gamma \colon [0,1] \to (K \Omega)^\circ$, 
there exist  
marked K\"ahler  K3 surfaces $(X,I_t,\sigma_t, \phi)$\footnote{
As the notation indicates, the complex structure $I_t$ and the K\"ahler
form $\sigma_t$ depend on $t \in [0,1]$, 
but the  manifold $X$ and the isometry $\phi$ do not. 
At this point, we do not insist that $\sigma_t$ depends smoothly
on $t$; see Lemma~\ref{pathforms}.}
satisfying
$$\tau_2(X,I_t,\sigma_t,\phi) = \gamma_t \mbox{ for all }t  \in [0,1];$$
moreover, $I_t$ depends smoothly on $t$.
\end{refTheorem}

\begin{proof}
By Theorem VIII.12.1 in \cite{BPV}, there exists a universal marked family  
$Y \stackrel{p}{\to} \calM_1$ of K3-surfaces. 
In particular,  given $x \in \calM_1$, the preimage $Y_s := p\inv(s)$ is a marked K3 surface.
The base space  $\calM_1$ 
is not Hausdorff,  but otherwise is a smooth analytic space of dimension $20$.
The period map $\tau_1 \colon \calM_1 \to \Omega$
is a local isomorphism by  \cite[Theorem VIII.7.3]{BPV}.

By Lemma VIII.9.3 in \cite{BPV}, 
$\calM'_2 := \cup_{s \in \calM_1} H^{1,1}(Y_s) \stackrel{\pi}{\to} \calM_1$
is a real analytic vector bundle.  
As described in \cite[\S VIII.12]{BPV},
the set of K\"ahler classes $\calM_2$ in $\calM_2'$
is an open submanifold.
Thus, $\calM_2$ is  real analytic manifold of dimension $60$,
and  the refined period map 
$\tau_2 \colon \calM_2 \to (K \Omega)^\circ$ is a real analytic submersion.
Since $\tau_2$ is both injective \cite[Theorem VIII.12.3]{BPV}
and surjective \cite[Theorem VIII.14.1]{BPV}, 
it induces a diffeomorphism from $M_2$ to $(K\Omega)^\circ$.

Given a smooth path $\gamma \colon [0,1] \to (K\Omega)^\circ$,
let $\gamma_t = (\kappa_t,[\alpha_t]).$
Since $\pi \circ \tau_2\inv \circ \gamma$ is a smooth path in $\calM_1$,
there exists marked K3 surfaces $(X,I_t,\phi)$ satisfying
$\tau_1(X,I_t,\phi) = [\alpha_t] \in \Omega$ for all $t \in [0,1]$.
Since $\tau_2 \inv \circ \gamma$ is a path in $\calM_2$,
there exists a K\"ahler form $\sigma_t \in \Omega^2(X)$ with
$\phi[\sigma_t] = \kappa_t$ for all $t \in [0,1]$.
\end{proof}

This has the following implication for tame K3 surfaces.

\begin{Corollary}\labell{existstamek3}
Given a smooth path $\gamma \colon [0,1] \to (\Tilde{K \Omega})^\circ$, there exist 
marked tame K3 surfaces $(X,I_t,\sigma_t,\phi)$  with
$$\tau_2(X,I_t,\sigma_t,\phi) = \gamma_t \mbox{  for all } t \in [0,1];$$
moreover, $I_t$ depends smoothly on $t$.
\end{Corollary}

\begin{proof}
Let $\Hat\gamma_t = (\Hat\kappa_t, [\alpha_t])$ for all $t \in [0,1]$,
where $\gamma_t = (\kappa_t,[\alpha_t])$ and
$\Hat\kappa_t$ is the projection of $\kappa_t$ onto $\alpha_t^\perp \cap L_\R$.
Then 
$\Hat\gamma_t$
is a smooth path in $(K \Omega)^\circ$ 
by \eqref{compare}.  
By Theorem~\ref{existsk3}, there exists marked K\"ahler K3 surfaces $(X,I_t,\Hat\sigma_t,\phi)$
satisfying $\tau_2(X,I_t,\sigma_t, \phi) = (\Hat\kappa_t, [\alpha_t])$ for all $t \in [0,1]$;
moreover, $I_t$ depends smoothly on $t$.
Hence, the claim follows by Lemma~\ref{tamedKahler}.

\end{proof}

We have now classified marked tame
K3 surfaces up to marked biholomorphism that preserve the tame class.
To classify them up to marked symplectomorphism, we need to
study the topology of $(\Tilde{K \Omega})^\circ$; c.f. \cite[\S VIII.9]{BPV}.

\begin{Lemma} \labell{TildeKOmega} 
Fix $\kappa \in L_\R$ with $(\kappa, \kappa) > 0$.  
\begin{enumerate}
\item
The subspace $(\Tilde{K \Omega})^\circ \subset \Tilde{K \Omega}$ is an open submanifold with two connected components that
$\Pi$ take  to different components of $G_3^+(L_\R)$.
\item
The intersection $\left( \{\kappa\} \times \Omega  \right) \cap (\Tilde{K \Omega})^\circ$
is a submanifold with two connected components
that $\Pi$ take  to different components of $G_3^+(L_\R)$.
\end{enumerate}
\end{Lemma}

\begin{proof}
The map from $\Omega$  to $G_2^+(L_\R)$ that sends
$[\alpha]$ to the oriented $2$-plane with oriented basis $\{ \Re \alpha, \Im \alpha \}$ is a 
diffeomorphism, and so we may identify these spaces.  
By Lemma~\ref{compare}, 
under this identification 
$\Tilde{K \Omega}$ is the set  
of pairs $(\kappa,W) \in L_\R \times G_2^+(L_\R)$
that together span a $3$-plane on which $( \cdot, \cdot)$ is positive definite.
Given $V \in G_3^+(L_\R)$, 
the preimage $\Pi\inv(V) \subset \Tilde{K\Omega}$ consists of 
$(\kappa, W) \in V \times G_2^+(V)$ with $\kappa  \not\in W$ such that
$\kappa$ and $W$ together induce the given orientation on $V$.
If we fix a non-zero $\kappa \in V$,
the set of such $W$ is diffeomorphic to a $2$-dimensional disk.
Thus $\Tilde{K \Omega}$ is a fiber bundle over $G_3^+(L_\R)$
with connected fibers.

We obtain $G_3^+(L_\R)^\circ$ from $G_3^+(L_\R)$ by removing 
$G_3^+(d^\perp)$ for all $d \in L$ with $(d,d) = -2$. 
Given any  $d \in L$, the set $G_3^+(d^\perp)$ is a closed submanifold
of $G_3^+(L_\R)$ of  codimension $3$.
Moreover, by the proof of \cite[Corollary VIII.9.2]{BPV}, 
the collection of $G_3^+(d^\perp)$ for $d \in L$ with $(d,d) = -2$ is locally
finite.
Therefore, $G_3^+(L_\R)^\circ$ is an open submanifold 
with two connected components, one in each component of $G_3^+(L_\R)$.
The first claim follows immediately.

Now fix  $\kappa \in L_\R$ with $(\kappa,\kappa) > 0$.
Given $W \in G_2^+(L_\R)$ with $(\kappa,W) \in \Tilde{K\Omega}$,
it is clear that $\kappa \in \Pi(\kappa,W)$.
Conversely, as we saw above, if $V \in G_3^+(L_\R)$
contains $\kappa$ then
$$ \big\{  W \in G_2^+(L_\R) 
\ \big| \ 
(\kappa, W) \in   \Tilde{K\Omega}  \mbox{ and }
\Pi(\kappa,W) = V \big\}$$
is diffeomorphic to a $2$-dimensional disk.
Thus,  $\left(\{\kappa\} \times G_2^+(L_\R)\right) \cap  \Tilde{K \Omega}$ is a fiber bundle over
 $\{V \in G^+_3(L) \mid \kappa \in V \}$ with contractible fibers.

Since the signature of $\kappa^\perp$ is $(2,19)$, the manifold
$G_2^+(\kappa^\perp)$ has two components.
Moreover, the natural diffeomorphism from $G_2^+(\kappa^\perp)$
to $\{ V \in G_3^+(L_\R) \mid \kappa \in V\}$ takes these components to different
components of $G_3^+(L_\R)$.
Given $d \in L$, the intersection
$$\{ V \in G_3^+(L_\R) \mid \kappa \in V\} \cap G_3^+(d^\perp)$$ 
is empty if $(\kappa,d) \neq 0$, and is
naturally diffeomorphic to $G_2^+(\kappa^\perp \cap d^\perp)$ if $(\kappa,d) = 0$.
Therefore, it is a closed submanifold of codimension $2$.
Since we have already seen that the collection of $G_3^+(d^\perp)$ for $d \in L$ with $(d,d) = -2$ is a locally finite collection,
this implies that $\{V \in G_3^+(L_\R)^\circ \mid \kappa \in V\}$
is a submanifold whose intersection with either component of $G_3^+(L_\R)$
has one connected component.
The second claim follows immediately.

\end{proof}

We also need a few general lemmas about tamed symplectic forms.

\begin{Lemma}\labell{biholomorphic}
Let $(X,I)$ and $(X',I')$ be closed complex manifolds
with tamed symplectic forms
$\sigma \in \Omega^2(X)$ and $\sigma' \in \Omega^2(X')$.
Given a biholomorphism $g$ from $(X,I)$ to $(X',I')$ with 
$g^*([\sigma']) = [\sigma]$,
there is a symplectomorphism $f$ from $(X,\sigma)$  to $(X',\sigma')$
with $f^* = g^* \colon H^*(X';\Z) \to H^*(X;\Z)$.
\end{Lemma}

\begin{proof}
The complex structure $I$ tames  $g^*(\sigma')$ and $\sigma$, 
and so
$\sigma_t := t g^*(\sigma') + (1-t) \sigma$ is a tamed symplectic form
for all $t \in [0,1]$. 
Thus, since $[g^*(\sigma')] = [\sigma]$ by assumption,
the  claim follows by Moser's method.
\end{proof}

\begin{Lemma}\labell{pathforms}
Let $(X,I_t)$ be a closed complex manifold with  tamed symplectic 
form $\sigma_t \in \Omega^2(X)$
for all $t \in [0,1]$. Assume that $I_t$ depends smoothly on $t$,
and that $[\sigma_t]  \in H^2(X;\R)$ is independent of $t$. 
Then there is a symplectomorphism  from $(X,\sigma_0)$ to $(X,\sigma_1)$ 
that induces the identity map on $H^*(X;\Z)$.
\end{Lemma}

\begin{proof}
Fix $s \in [0,1]$.
Since $I_s$ tames $\sigma_s$,  $I_t$ depends smoothly on $t$, 
and $X$ is closed, there exists an open  neighborhood
$V_s$ of $s \in [0,1]$ so that
$I_t$ tames $\sigma_s$ for all $t \in V_s$.
Choose
a partition of unity $\{\rho_{s}\}$ subordinate to $\{V_{s}\}$ so that
$\rho_0(0) = 1$  and $\rho_1(1) = 1$.
Then $\sigma'_t := \sum_s \rho_{s}(t) \sigma_{s}$ is tamed by
$I_t$ (and hence is symplectic) for all $t \in [0,1]$.
Since $\sigma'_i = \sigma_i$ for $i = 0,1$,
the  claim now follow by Moser's method.
\end{proof}

We are now ready to classify  marked tame K3 surfaces 
that are taken by $\Pi \circ \tau_2$ to a given component of $G_3^+(L_\R)$,
up to marked symplectomorphism.

\begin{Lemma}\labell{classifymk3} \
\begin{enumerate}
\item Given $\kappa \in L_\R$ and a component of $G_3^+(L_\R)$, there exists
a marked tame K3 surface $(X,I,\sigma,\phi)$ with $\phi\big([\sigma]\big) = \kappa$
and $\Pi(\tau_2(X,I,\sigma,\phi))$ in the given component if and only if
$(\kappa,\kappa) > 0$.
\item Given marked tame K3 surfaces $(X_0,I_0,\sigma_0,\phi_0)$ and
$(X_1,I_1,\sigma_1,\phi_1)$ so that 
$\Pi(\tau_2(X_0,I_0,\sigma_0,\phi_0))$
and $\Pi(\tau_2(X_1,I_1,\sigma_1,\phi_1)$ lie in the same
component of $G_2^+(L_\R)$, 
there is a marked symplectomorphism from
$(X_0,\sigma_0,\phi_0)$ to $(X_1,\sigma_1,\phi_1)$ if and only
if 
$\phi_0\big([\sigma_0] \big) = \phi_1 \big([\sigma_1] \big)$.
\end{enumerate}
\end{Lemma}

\begin{proof}
Fix $\kappa \in L_\R$. Since the cup product pairing is positive on tamed
classes,  we may assume that $(\kappa,\kappa) > 0$.
Hence, claim (1) follows immediately from
Lemma~\ref{TildeKOmega}.(2)
and  Corollary~\ref{existstamek3}.

Let $(X_0,I_0,\sigma_0,\phi_0)$ and
$(X_1,I_1,\sigma_1,\phi_1)$ be marked tame K3 surfaces so that
$\Pi(\tau_2(X_0,I_0,\sigma_0,\phi_0))$
and $\Pi(\tau_2(X_1,I_1,\sigma_1,\phi_1)$ lie in the same
component of $G_3^+(L_\R)$; assume that
$\phi_0\big([\sigma_0] \big) = \phi_1 \big([\sigma_1] \big)$.

By Lemmas~\ref{image} and  \ref{TildeKOmega}.(2), there is a smooth path 
$\gamma  \colon [0,1] \to
\left( \big\{ \phi_0\big([\sigma_0]\big) \big\} \times \Omega \right) \cap (\Tilde{K\Omega})^\circ$
with
\begin{equation}\labell{mp1}
\tau_2(X_i,I_i,\sigma_i,\phi_i) = \gamma_i  
\mbox{  for }i = 0,1.
\end{equation}
By Corollary~\ref{existstamek3}, there exists  marked tame K3 surfaces 
$({X}',{I}'_t,\sigma'_t,{\phi}')$ with 
\begin{equation}\labell{mp2}
\tau_2({X}',{I}'_t,{\sigma}'_t,{\phi}') = \gamma_t \mbox{  for all } t \in [0,1];
\end{equation}
moreover, $I'_t$ depends smoothly on $t$.
Hence, by Lemma~\ref{pathforms} implies that there is a marked symplectomorphism
$f'$ from $(X',\sigma'_0,\phi')$ to $(X',\sigma'_1,\phi')$.
By Corollary~\ref{tamedtorelli},
\eqref{mp1} and \eqref{mp2}
imply that there exist marked biholomorphism $g_i$ from $(X_i,I_i, \sigma_i,\phi_i)$
to $(X',I'_i, \sigma'_i, \phi')$ for $i = 0,1$.
Thus,  by Lemma~\ref{biholomorphic}  
there is marked symplectomorphisms $f_i$ from $(X_i,\sigma_i, \phi_i)$
to $(X',\sigma'_i, \phi')$ for $i = 0,1$.
Then  $ f_1\inv \circ f' \circ f_0$ is a marked 
symplectomorphism from $(X,\sigma_0,\phi_0)$
to $(X,\sigma_1,\phi_1)$.
This proves claim (2).
\end{proof}

Technically, we could end Section~\ref{ss:K3} here and still prove Theorem~\ref{main}.
More precisely, we could prove 
Proposition~\ref{freeunique2} 
and 
Proposition~\ref{existsfree2} 
by using parts (2) and (1) of 
Lemma~\ref{classifymk3}, respectively, 
instead of the corresponding parts of Proposition~\ref{classifyk3}.
On the other hand, the theorem below -- which is a slight reformulation
of a theorem of Donaldson --  immediately leads to a proof
of Proposition~\ref{classifyk3}.
Additionally, it 
completes the classification
of marked tame K3 surfaces started in Lemma~\ref{classifymk3}
by showing  that there are no marked symplectomorphisms between
two marked tame K3 surfaces if $\Phi \circ \tau_2$ takes them
to different components of $G_3^+(Li_\R)$.

\begin{refTheorem}\labell{lift}
Let $(X_0,I_0,\sigma_0,\phi_0)$ and $(X_1,I_1,\sigma_1, \phi_1)$
be marked tame K3 surfaces.
Then there is a marked diffeomorphism from $(X_0,\phi_0)$ to
$(X_1,\phi_1)$ if and only if
$\Pi(\tau_2(X_0,I_0,\sigma_0,\phi_0))$ and $\Pi(\tau_2(X_1,I_1,\sigma_1,\phi_1))$
lie in the same component of $G_3^+(L_\R)$.
\end{refTheorem}

\begin{proof}
Assume first that
$\Pi(\tau_2(X_0,I_0,\sigma_0,\phi_0))$ and $\Pi(\tau_2(X_1,I_1,\sigma_1,\phi_1))$
lie in the same component of $G_3^+(L_\R)$.
By Lemmas~\ref{image} and  \ref{TildeKOmega}.(1), there is a smooth path 
$\gamma \colon [0,1] \to (\Tilde{K\Omega})^\circ$  with
\begin{equation*}
\tau_2(X_i,I_i,\sigma_i,\phi_i) = \gamma_i  
\mbox{  for }i = 0,1.
\end{equation*}
Hence,  by Corollaries~\ref{existstamek3} and  
\ref{tamedtorelli},
there exists  marked tame K3 surfaces 
$({X}',{I}'_t,\sigma'_t,{\phi}')$  for all $t \in [0,1]$, and
marked biholomorphisms $g_i$ from $(X_i,I_i, \sigma_i,\phi_i)$
to $(X',I'_i, \sigma'_i, \phi')$ for $i = 0,1$.
Then $g_1\inv \circ g_0$ is a  marked diffeomorphism from $(X,\phi_0)$
to $(X,\phi_1)$.  

Now assume that there is a marked diffeomorphism from $(X_0,\phi_0)$ to $(X_1,\phi_1)$. 
Donaldson proves that there is no diffeomorphism
$f \colon X_0 \to X_0$ with 
$f^* \colon H^2(X_0;\Z) \to H^2(X_0;\Z) = - {\bf 1}$ \cite{Do}.
Hence, there is no  marked diffeomorphism from $(X_0,\phi_0)$
to $(X_1,-\phi_1)$. By the previous paragraph, this implies that
$\Pi(\tau_2(X_0,I_0,\sigma_0,\phi_0))$ and $\Pi(\tau_2(X_1,I_1,\sigma_1,-\phi_1))$
lie in different components of $G_3^+(L_\R)$, i.e.,
$\Pi(\tau_2(X_0,I_0,\sigma_0,\phi_0))$ and $\Pi(\tau_2(X_1,I_1,\sigma_1,\phi_1))$
lie in the same component.
\end{proof}

\begin{proof}[Proof of Proposition~\ref{classifyk3}]
Fix $\kappa \in H^2(X;\R)$.
As in Lemma~\ref{classifymk3}, we may assume that $(\kappa,\kappa) > 0$.
By Siu's theorem, there exists a marked tame K3 structure $(I,\sigma,\phi)$ on $X$.
By  Lemma~\ref{classifymk3}.(1), 
there exists
a marked tame K3   surface
$(X',{I}',\sigma',{\phi}')$  with 
$\phi'\big([\sigma']\big) = \phi\big([\kappa]\big)$ so that
$\Pi(\tau_2(X,I,\sigma,\phi))$ and $\Pi(\tau_2(X,I',\sigma',\phi'))$ lie
in the same component of $G_3^+(L_\R)$.
Thus,  Theorem~\ref{lift} implies that there exists a marked diffeomorphism
$f$ from  $(X,\phi)$ to  $({X}, {\phi}')$. 
The pullback $(X,f^*({I}'),f^*({\sigma}'))$
is a tame K3 surface with
$\phi \big( [f^*({\sigma}')] \big) = \phi' \big( [\sigma'] \big) 
= \phi(\kappa)$. This proves claim  (1).

Let $(I_0,\sigma_0)$ and $(I_1,\sigma_1)$ be tame K3 structures on $X$,
and consider an isometry $\phi \colon H^2(X;\Z) \to H^2(X;\Z)$.
Given a marking $\phi_0 \colon H^2(X;\Z) \to L$, 
let $\phi_1 = \phi_0 \circ \phi$.
Then
$\Pi(\tau_2(X,I_0,\sigma_0,\phi_0))$
and $\Pi(\tau_2(X,I_0,\sigma_0,\phi_1))$ lie in the same component of
$G_3^+(L_\R)$ exactly if
$\phi$ preserves the components of $G_3^+(L_\R)$, 
Hence, Theorem~\ref{lift} implies that 
that there is no diffeomorphism $f$ 
with $f^* = \phi$
if $\phi$ reverses the components.\footnote{
Alternatively, this is the original statement of Donaldson's theorem \cite{Do}.} 
On the other hand, 
Theorem~\ref{lift} implies that
$\Pi(\tau_2(X,I_0,\sigma_0,\phi_1))$
and $\Pi(\tau_2(X,I_1,\sigma_1,\phi_1))$ lie in the same component of $G_3^+(L_\R)$.
Hence, Claim (2)
follows immediately from Lemma~\ref{classifymk3}.(2).
\end{proof}

\section{Free Hamiltonian actions}
\labell{ss:free}

In this section, we analyze free Hamiltonian circle actions
on symplectic manifolds with  reduced spaces symplectomorphic 
to tame K3 surfaces (that also satisfy a technical  condition).
In this case,
the  Duistermaat-Heckman function  
is a positive polynomial of degree at most two with even coefficients.
Our main result is that
these polynomials classify such 
symplectic manifolds.

\begin{Proposition}\labell{freeunique2}
Let the circle act freely on symplectic manifolds
$(M,\omega)$ and $(M',\omega')$ with proper moment maps
$\Psi \colon M \to (a,b)$ and $\Psi' \colon M' \to (a,b).$
Assume that, for all $t \in (a,b),$  the reduced spaces 
$M \mod_{\! t} \, S^1 $ and 
$M' \mod_{\! t} \, S^1$ are  
symplectomorphic to tame K3 surfaces $(X,I_t,\sigma_t)$\footnote{
By Ehresmann's lemma, we can (and do) choose the symplectomorphisms
$M \mod_{\! t}\, S^1 \stackrel{\simeq}{\to} X$ to induce a diffeomorphism 
$M/S^1 \stackrel{\simeq}{\to} X \times (a,b)$.  
Hence, the forms $\sigma_t$ depend smoothly on $t$.
However, the complex structures $I_t$ may not.}
and $(X',I_t',\sigma_t')$; moreover,
\begin{itemize}
\item $(\sigma_t,\sigma_t) = (\sigma'_t,\sigma'_t)$; and
\item $[\sigma_t] = \kappa - t \eta$ and $[\sigma'_t] = \kappa' - t \eta'$, where 
$\kappa, \eta $ and $\kappa',\eta'$ 
induce primitive embeddings $\Z^2 \hookrightarrow H^2(X;\Z)$ and $\Z^2 \hookrightarrow H^2(X';\Z)$.
\end{itemize}
Then every $t \in (a,b)$ has a neighborhood $U$ so that
$\Psi\inv(U)$ and $(\Psi')\inv(U)$ are equivariantly symplectomorphic.

\end{Proposition}

In the situation described in Proposition~\ref{freeunique2},
 the Duistermaat-Heckman function of $M$
(and $M'$) at $t \in (a,b)$  is $ (\kappa,\kappa) - 2 t (\kappa, \eta) + t^2 (\eta,\eta)$.
Since the K3 lattice is even and $\kappa$ and $\eta$ are integral, this is 
a polynomial of degree at most two with even coefficients; it is positive on $(a,b)$.

\begin{Proposition}\labell{existsfree2}
Fix  a polynomial $P$ of degree at most two with even coefficients  that
is positive on $[a,b] \subset \R$.
Then there exists a free circle action on a symplectic
manifold $(M,\omega)$ with proper moment map $\Psi \colon M \to (a , b)$
so that, for all $t \in (a, b)$, 
the reduced space
$M \mod_{\! t} \, S^1$ is symplectomorphic to a tame
K3 surface $(X,I_t,\sigma_t)$; moreover,
\begin{itemize}
\item $(\sigma_t, \sigma_t) = P(t);$ and
\item $[\sigma_t] = \kappa - t \eta$, where $\kappa, \eta$  induce a primitive embedding $\Z^2 \hookrightarrow H^2(X;\Z)$.
\end{itemize}

\end{Proposition}
To prove these propositions,
we need  two standard lemmas on the uniqueness and existence
of free Hamiltonian circle actions; see, for example, \cite{MS}.

\comment{Please check the citation above.}

\begin{Lemma}\labell{Hamunique}
Let the circle act freely on symplectic manifolds
$(M,\omega)$ and $(M',\omega')$ with proper moment maps
$\Psi \colon M \to (a,b)$ and $\Psi' \colon M' \to (a,b),$
and fix $t \in (a,b)$.
Let  $f \colon 
M \mod_{\! t} \, S^1  \to M' \mod_{\! t} \, S^1$  
be a symplectomorphism
so that  $f^*$
takes the Euler class of the circle bundle
$(\Psi')\inv(t) \to M' \mod_{\! t}\, S^1$ to the Euler class of
$\Psi\inv(t) \to M \mod_{\! t}\, S^1$.
Then
$\Psi\inv(U)$ and $(\Psi')\inv(U)$ are equivariantly symplectomorphic
for some neighborhood $U$ of $t$. 
\end{Lemma}

\begin{Lemma}\labell{Hamexists}
Let $(X,\sigma)$ be a closed symplectic manifold; fix $\mu \in \Omega^2(X)$ with $[\mu] \in H^2(X;\Z)$.
There exists $\epsilon > 0$, and 
a free circle action
on a symplectic manifold $(M,\omega)$  with
proper moment map $\Psi \colon M \to (-\epsilon, \epsilon)$
so that
the reduced space $M \mod_{\! t} \, S^1$ is 
symplectomorphic  to $(X,\sigma - t \mu)$ for all $t \in (-\epsilon, \epsilon)$.
\end{Lemma}

We also need the following fact.

\begin{Lemma}\labell{iso}
Fix $\kappa, \eta, \kappa',$ and $\eta'$ in the K3 lattice  $L$ satisfying:
\begin{itemize}
\item $(\kappa - t \eta,  \kappa - t \eta ) = (\kappa' - t \eta' ,\kappa' - t \eta')$ for all $t \in \R$; and
\item  each pair $\kappa, \eta$ and $\kappa', \eta'$ induces a primitive embedding $\Z^2 \hookrightarrow L$.
\end{itemize}
Then there exists an isometry of $L$ that takes $\kappa'$ to $\kappa$,
takes $\eta'$ to $\eta$, and preserves 
(alternatively, reverses) the components of $G_3^+(L_\R)$.
\end{Lemma}

\begin{proof}
The K3 lattice  can be written as the orthogonal direct sum
$$L \simeq H \oplus H \oplus H \oplus -E_8 \oplus -E_8,$$
where $H$ is the indefinite, even, unimodular lattice of rank $2$;
and $E_8$ is the positive definite, even, unimodular lattice of rank $8$.
Hence, for $i \in \{1,2,3\}$,
there exist $e_i$ and $f_i$ in the $i$'th summand above
satisfying $(e_i, e_j) = (f_i , f_j) = 0$
and $(e_i,f_j) = \delta_{ij}$ for all $i,j \in \{1,2,3\}$. 
Consider
$$\Tilde{\kappa} := e_1 + \half (\kappa,\kappa) f_1  
\ \mbox{and} \ \Tilde{\eta} := (\kappa,\eta) f_1 + e_2 + \half (\eta,\eta) f_2 \ \in L.$$
The isometry $\phi \colon H^2(X;\Z) \to H^2(X;\Z)$ 
that takes $e_3$ to $-e_3$ and $f_3$ to $-f_3$ but is  the identity map on the other summands
fixes $\Tilde{\kappa}$ and $\Tilde{\eta}$ but exchanges the components of
$G^+_3(L)$.
By assumption, the pairs $\kappa, \eta$; 
$\kappa',\eta'$; and  $\Tilde{\kappa}, \Tilde{\eta}$ each  induce a primitive embedding $\Z^2 \hookrightarrow L$.
Moreover, their images are isomorphic  lattices.
Hence,  \cite[Theorem I.2.9]{BPV} implies that  there exists an isometry of $L$
that takes  $\kappa'$ to $\Tilde{\kappa}$ and $\eta'$ to $\Tilde{\eta}$, 
and another  isometry that takes  $\Tilde{\kappa}$ to $\kappa$ and $\Tilde{\eta}$ to 
$\eta$.
By either composing these two isometries or  composing these two isometries with 
$\phi$ inserted between them,
we construct the required isometry.
\end{proof}

Using the results (and notation) from the previous section, we can now specialize to that case
that the reduced spaces are symplectomorphic to tame K3 surfaces.

\begin{proof}[Proof of Proposition~\ref{freeunique2}]
Fix $t \in (a,b)$.
Since all K3 surfaces are diffeomorphic, we may assume that $X' = X$.
Since $H^2(X;\Z) \simeq L$,  Lemma~\ref{iso}
implies that there is an isometry $\phi \colon H^2(X;\Z) \to H^2(X;\Z)$
so that $\phi(\kappa') = \kappa$, $\phi(\eta') =\eta$,
and $\phi$ preserves the 
components of $G_3^+(H^2(X;\R))$.
By Proposition~\ref{classifyk3}.(2), this implies that there exists a 
symplectomorphism $f$ from $(X,\sigma_t)$ to $(X',\sigma_t')$
satisfying
 $f^*(\eta')  = \eta$. 
Since $H^2(X;\Z)$ is torsion-free,
$\eta \in H^2(X;\Z)$ is the Euler class of
the circle bundle $\Psi\inv(t) \to M \mod_{\! t} \, S^1$ and
$\eta' \in H^2(X;\Z)$ is
the Euler class
of 
$(\Psi')\inv(t) \to M' \mod_{\! t}\, S^1$.
Hence, the claim follows immediately from Lemma~\ref{Hamunique}.
\end{proof}

\begin{proof}[Proof of Proposition~\ref{existsfree2}]
Write $P(t) = 2 \ell_2 t^2 + 2 \ell_1 t + 2 \ell_0 $, and let
$X$ be a manifold that admits a K3 structure.
Since $H^2(X;\Z) \simeq L$,  there exist $e_1, e_2, f_1 $ and $f_2 \in H^2(X;\Z)$ 
satisfying $(e_i,e_j) =  (f_i,f_j) = 0$ and $(e_i,f_j) = \delta_{ij}$ 
for all $i,j \in \{1,2\}$.
Consider $\kappa := e_1 + \ell_0 f_1$ and $\eta := - \ell_1 f_1 + e_2 + \ell_2 f_2\in H^2(X;\Z)$.
Then  $(\kappa - t \eta, \kappa - t \eta) = P(t)$ for all $t$,
and $\kappa, \eta$ induce a primitive embedding $\Z^2 \hookrightarrow H^2(X;\Z)$.

Fix $t \in [a,b]$.  
Since $P(t)> 0$,
Lemma~\ref{classifyk3} implies that there exists a tame K3  structure
$(I,\sigma)$ on $X$ satisfying $[\sigma] = \kappa - t \eta$. 
Pick a closed two-form $\mu \in \Omega^2(X)$ with $[\mu] = \eta$.
By Lemma~\ref{Hamexists}, 
there exists $\epsilon > 0$, and a free circle action on a symplectic manifold $(M, \omega)$ with proper moment map  
$\Psi \colon M \to (t -\epsilon, t+ \epsilon)$ so that
the reduced space 
$M \mod_{\! s} \, S^1$  is symplectomorphic 
to  $(X,\sigma - (s-t) \mu)$ for all $s \in (t -\epsilon,t+ \epsilon)$.
By construction, 
$$[\sigma - (s - t) \mu] = \kappa - s \eta \quad \forall s \in (t - \epsilon, t + \epsilon).$$
Finally, since tameness is an open condition, after possible shrinking $\epsilon$,
$I$ tames $\sigma - (s - t)\mu$ for all $s \in (t - \epsilon, t + \epsilon)$.

Since $[a,b]$ is compact,
this implies that we can cover $[a,b]$ by open sets
$V_1,\dots,V_k$ so that each $V_i$ is the moment image of
a circle action on a symplectic manifold $(M_i,\omega_i)$ 
satisfying the prescribed conditions.
By Proposition~\ref{freeunique2}, after possibly shrinking the $M_i$ and $V_i$, we may assume that $M_i$ is equivariantly symplectomorphic to $M_j$ over $V_i\cap V_j$ for all $i$ and $j$. Therefore, we can construct $M$ by
gluing together the $M_i$.
\end{proof}

\subsection{Kotschick's theorem}
\label{Kotschick} 
As we mentioned in the introduction, our proof is adapted from \cite{Ko}.
Nevertheless, we  now give a simple proof of his main theorem.

\begin{refTheorem}[Kotschick]
There exists a free symplectic circle action on a closed connected 
six-dimensional symplectic manifold so that each orbit  is 
contractible.
\end{refTheorem}

\begin{proof}
Fix $\epsilon > 0$.
By Proposition~\ref{existsfree2}, there exists 
a free circle action on a symplectic
manifold $(M,\omega)$ with proper moment map $\Psi \colon M \to (-\epsilon , 1 + \epsilon)$
so that, for all $t \in (- \epsilon, 1 + \epsilon)$, 
the reduced space
$M \mod_{\! t} \, S^1$ is symplectomorphic to a tame
K3 surface $(X,I_t,\sigma_t)$; moreover,
\begin{itemize}
\item $(\sigma_t, \sigma_t) = 2;$ and
\item 
$[\sigma_t] = \kappa - t \eta$, 
where  $\kappa, \eta$ induce a primitive embedding
$\Z^2 \hookrightarrow  H^2(X; \Z)$.  
\end{itemize}
Since $H^2(X;\Z)$ is torsion-free,
$\eta$ is the Euler class of
the circle bundle $\Psi\inv(t) \to M \mod_{\! t} \, S^1$ for all $t \in (-\epsilon, 1+ \epsilon)$.
Since $\eta$ is primitive and $X$ is simply connected, 
this implies that the orbit $S^1 \cdot x$ is
contractible in $\Psi\inv(t)$ for all $x \in \Psi\inv(t)$.
Finally, Proposition~\ref{freeunique2} implies that, after possibly shrinking $\epsilon$,
$\Psi^{-1}(-\epsilon, \epsilon)$ and
$\Psi^{-1}(1-\epsilon,1+\epsilon)$ are equivariantly symplectomorphic. 
We  construct the proposed manifold by identifying these subspaces.
\end{proof}

\section{Locally free Hamiltonian actions}
\labell{ss:locally free}

In this section, we consider locally free
Hamiltonian circle actions on symplectic manifolds with reduced 
spaces diffeomorphic to the  Kummer surface $T/\Z_2$.
(Recall that  $\Z_2$ acts on  
$T = \C^2/\left(\Z^2 + \sqrt{-1}\,  \Z^2 \right)$  
by  the involution $[z] \mapsto [-z]$.)
Unlike the previous section, we don't prove general theorems but  simply construct the
examples we need.
Moreover,  we endow our examples with congenial complex structures 
that we use to add  fixed points in the next section.
More precisely, we prove the following:

\begin{Proposition}\labell{existslf}
There exist complex manifolds $(M_+,J_+)$ and $(M_-, J_-)$ with 
locally free holomorphic $\C^\times$ actions,
 $S^1 \subset \C^\times$  invariant  symplectic forms $\omega_{\pm} \in \Omega^2(M_{\pm})$, 
proper moment maps  $\Psi_{\pm} \colon M_{\pm} \to \R$, and $\C^\times$ invariant maps
$\pi_\pm \colon M_\pm \to T/\Z_2$ 
so that the following hold:
\begin{enumerate}
\item For all $t \in \R$, the map $\pi_\pm$ induces
a symplectomorphism from 
$M_\pm \mod_{\! t}\, S^1$ to $T/\Z_2$ with 
$(\sigma_\pm)_t \in \Omega^2(T/\Z_2)$, where
\begin{equation}\labell{sigmapm}
(\sigma_\pm)_t =
dz_1 \, dz_2 + d\zbar_1 \, d \zbar_2 \pm  \sqrt{-1} \, t dz_1 \,d\zbar_1 \pm   \sqrt{-1} \, t dz_2 \, d \zbar_2.\footnote{ 
By a slight abuse of notation, we describe a form on 
$T/\Z_2$
by giving its pull-back to $T$.} 
\end{equation}
\item $\pi_\pm$  induces 
a biholomorphism from $M_\pm/\C^\times$ to $T/\Z_2$. 
\item $\omega_{\pm}(\xi_\pm, J_{\pm}(\xi_\pm)) > 0$, where $\xi_\pm 
\in \chi(M_\pm)$ generates the $S^1$ action. 
\item $J_\pm$ tames $\omega_\pm$ on the  preimage
$\Psi_\pm\inv\big(\pm (0,\infty)\big)$.
\end{enumerate}
Moreover, there is an equivariant symplectomorphism from $M_+$ to $M_-$ that intertwines the moment maps.
\end{Proposition}

\begin{Remark}\labell{varphi}
The integral of the pull-back of $(\sigma_\pm)_t \wedge (\sigma_\pm)_t$ over $T$ is $8 + 8 t^2$.
Hence, the  Duistermaat-Heckman function of $M_\pm$ is $4 + 4 t^2$.
\end{Remark}

\begin{proof}
Fix integers $k_1$ and $k_2$ with $k_1 k_2 > 0$.
Identify $\C^\times$ with the quotient $\C/(\sqrt{-1} \, \Z)$, and hence
$S^1 \subset \C^\times$ with the quotient $\sqrt{-1} \, \R/( \sqrt{-1} \, \Z)$.
Under this identification, the group $\Z^4$ acts holomorphically on $\C^2 \times \C^\times$ by
\begin{multline*}
(n_1,m_1,n_2,m_2) \cdot (z_1,z_2; u) =  \\
\left(z_1 + n_1 + \sqrt{-1}\,  m_1, z_2 + n_2 + \sqrt{-1}\,  m_2; 
u  - k_1 (n_1 z_1 +  n_1^2 /2)  - k_2 (n_2 z_2 + n_2^2/2)  \right)
\end{multline*}
for all $(n_1,m_1,n_2,m_2) \in \Z^4$ and $(z_1,z_2;u) \in \C^2 \times \C^\times$.

The 
quotient $(\Tilde{M},\Tilde{J}) = \left( \C^2 \times \C^\times \right) /\Z^4$ is a complex manifold
with a free holomorphic $\C^\times \simeq \C/( \sqrt{-1} \, \Z)$  action given by
$$[w] \cdot [z_1,z_2;u] = [z_1, z_2 ; u + w]$$
for all $w \in \C^\times$ and $[z_1,z_2;u] \in \Tilde{M}$.
Let $\Re(x+ \sqrt{-1} \, y) := x$ for all $x,y \in \R$,
and define an $S^1$ invariant proper function 
$\Tilde{\Psi} \colon \Tilde{M} \to \R$ by
$$
\Tilde{\Psi}\big([z_1,z_2;u]\big) =    \Re(u) + k_1 \Re(z_1)^2/2 +  k_2 \Re(z_2)^2/2 $$ 
for all $[z_1,z_2;u] \in \Tilde{M}$.
Next, define a  closed real $S^1$ invariant $(1,1)$-form  on $\Tilde{M}$ by
$$
 \Tilde{\eta} :=  \sqrt{-1} \, \partial 
\overline{\partial} \big( {\Tilde{\Psi}^2} \big)  = 
2 \sqrt{-1} \, \del \Tilde{\Psi} \wedge \delbar \Tilde{\Psi} 
+ \sqrt{-1} \, \Tilde{\Psi} \left(
k_1 dz_1 \, d\overline{z}_1 +  k_2  dz_2 \,  d\overline{z}_2 \right)/ 2,$$
where
$$\del \Tilde{\Psi} = 
\left(d u  + k_1 \Re(z_1) dz_1 +   
k_2 \Re(z_2) dz_2 \right)/2.$$
The form $\Tilde{\eta}$ is not symplectic.
However, if $k_1$ and $k_2$ are positive (respectively, negative),
the complex structure $\Tilde{J}$ tames the form $\Tilde{\eta}$  
on the preimage $\Tilde{\Psi}^{-1}(0,\infty)$ 
(respectively,  $\Tilde{\Psi}^{-1}(-\infty,0) $).

Given a non-zero integer $k_3$,
define  a closed real $S^1$ invariant two-form $\Tilde{\omega}$ on $\Tilde{M}$ by
$$\Tilde{\omega} = k_3 d z_1 d z_2/2 + k_3 d \zbar_1  d \zbar_2/2 + \Tilde{\eta}.$$
Since 
$$
\Tilde{\omega}^3/6    
= -  \sqrt{-1}\,  (k_3^2 + k_1 k_2 \Tilde{\Psi}^2 ) d z_1 \, d \zbar_1 \, d z_2 \, d \zbar_2 du \, d \overline{u}/8 \neq 0, 
$$
the form $\Tilde{\omega}$ is  symplectic.
The action of the circle $S^1 \subset \C^\times$ on $\Tilde{M}$ is generated by the vector field
$$\Tilde{\xi} = \sqrt{-1} \,
\left( \frac{\partial}{\partial u} - \frac{\partial}{ \partial\ubar }\right),$$ 
and so
$\iota_{\Tilde{\xi}} \Tilde{\omega} =  - d \Tilde{\Psi}.$
Therefore, $\Tilde{\Psi} \colon \Tilde{M} \to \R$ is a moment map for the circle action
on $(\Tilde{M},\Tilde{\omega})$.
For any $t \in \R$,
the map $\Tilde{\pi}$ from $\Tilde{M}$ to $T$ given by $[z_1,z_2;u] \mapsto [z_1,z_2]$  
induces a diffeomorphism from the reduced space
$\Tilde{\Psi}\inv(t)/S^1$ to  $T$.
Since the form  $\del \Tilde{\Psi} \wedge \delbar \Tilde{\Psi} = \del \Tilde{\Psi} \wedge d \Tilde{\Psi}$
vanishes on the level set $\Tilde{\Psi}\inv(t)$,  
this identifies 
the reduced symplectic form with $\Tilde{\sigma}_t \in \Omega^2(T)$, where
$$\Tilde{\sigma}_t  = k_3 dz_1 dz_2/2  + k_3 d \zbar_1 d \zbar_2/ 2 + 
\sqrt{-1}\,  k_1 t dz_1 d \zbar_1/2  + \sqrt{-1} \, k_2 t d z_2 d \zbar_2/2.$$
Additionally, the map $\Tilde{\pi}$  induces a biholomorphism from 
$\Tilde{M}/\C^\times$ to $T$.
It is straightforward to check that  $\Tilde{\omega}\big({\Tilde{\xi}}, \Tilde{J}\big({\Tilde{\xi}} \,\big)\big) > 0$.
Finally, since $dz_1 \, d z_2 \in \Omega^{2,0}(\Tilde{M})$ and
$d \overline{z_1} \, d \overline{z}_2 \in \Omega^{0,2}(\Tilde{M})$,
the preceding paragraph implies that
if $k_1$ and $k_2$ are positive (respectively, negative) then the complex structure $\Tilde{J}$ tames the form $\Tilde{\omega}$  
on the preimage $\Tilde{\Psi}^{-1}(0,\infty)$ 
(respectively,  $\Tilde{\Psi}^{-1}(-\infty,0)$);
cf. Lemma~\ref{tamedKahler}.

Given an integer $\ell$, the  group  $\Z_2$ acts holomorphically
on $\Tilde{M}$  by the involution
\begin{equation}\labell{involution}
[ z;u] \mapsto [-z;u + \sqrt{-1}\,  \ell/2 ].
\end{equation}
In most cases, the action is not free and 
the quotient $\Tilde{M}/Z_2$ is a complex orbifold $(M_+,J_+)$.
However, if $k_1$ and $k_2$ 
are  even and $\ell$ is odd, this action is free,
and so $M_+$ is a complex manifold.
In any case, since the $\C^\times$ action on $\Tilde{M}$ commutes with the involution,  
it induces a locally free holomorphic  $\C^\times$ action  on $M_+$.
Since the involution  preserves 
the symplectic form $\Tilde{\omega}$ and the  moment map $\Tilde{\Psi}$, 
$M_+$ inherits  an $S^1$ invariant  symplectic form $\omega_+$ and proper moment map 
$\Psi_+$.
Finally, since  $\Tilde{\pi}$ is $\Z_2$ equivariant, it engenders a $\C^\times$ invariant map $\pi_+ \colon M_+ \to T/\Z_2$
which induces a diffeomorphism from the
reduced space $M_+ \mod_{\! t}S^1$ 
to the  orbifold
$T/Z_2$ for all $t \in \R$;
 this identifies the reduced symplectic form  on $M \mod_{\! t} S^1$ with 
$(\sigma_+)_t \in \Omega^2(T/\Z_2)$, where
\begin{equation}\labell{sigmaplus}
(\sigma_+)_t =
 k_3 d z_1 dz_2/2 + k_3 d \zbar_1 d \zbar_2/2 + 
\sqrt{-1}\, k_1 t dz_1 d \zbar_1/2 + \sqrt{-1}\, k_2 t dz_2 d \zbar_2/2.
\end{equation}
Clearly, claims (2) and (3) are satisfied;
if  $k_1$ and $k_2$ are positive, then claim (4) is satisfied as well.

Now  repeat the process with $-k_1, -k_2, k_3,$ and $\ell$ to construct
another complex orbifold $(M_-,J_-)$ with  locally free holomorphic 
$\C^\times$ action,    $S^1$ invariant  symplectic form 
$\omega_-$,  proper  moment map $\Psi_- \colon M_- \to \R$, and
$\C^\times$ invariant map $\pi_- \colon M_- \to T/\Z_2$.
Then, for all $t \in \R$, the map $\pi_-$ induces a symplectomorphism from $M_- \mod_{\! t} \, S^1$ to $T/\Z_2$ with $(\sigma_-)_t \in \Omega^2(T/\Z_2)$, where
\begin{equation}\labell{sigmaminus}
(\sigma_-)_t =
 k_3 d z_1 dz_2/2 + k_3 d \zbar_1 d \zbar_2/2  
- \sqrt{-1}\, k_1 t dz_1 d \zbar_1/2 - \sqrt{-1}\, k_2 t dz_2 d \zbar_2/2.
\end{equation}
Again, if $k_1$ and $k_2$ are positive then claims (2), (3) and  (4) are satisfied.

Define an equivariant  diffeomorphism $f \colon M_+ \to M_-$ that intertwines the moment maps by 
$$f([z_1,z_2;u]) = [\zbar_1,  \zbar_2; u + k_1 \Re(z_1)^2 + k_2 \Re(z_2)^2].$$
Working in coordinates, it is straightforward to check that 
$f$ is a  symplectomorphism.

The proposition now follows if we take $k_1 = k_2 = k_3 = 2$.
\end{proof}

\begin{Remark}
More generally, 
given integers $k_1, k_2 > 0$ and $k_3 \neq 0$,
the proof above shows that Proposition~\ref{existslf} still  holds
if we replace ``manifolds" by ``orbifolds" and replace \eqref{sigmapm}
by \eqref{sigmaplus} and \eqref{sigmaminus}.
Given an integer $\ell$, the point $[z;u] \in \Tilde{M}$ 
is fixed  by the involution \eqref{involution} exactly if
\begin{enumerate}
\item [(a)] $z_i =  \left (m_i + \sqrt{-1}\, n_i \right)\!/2$, 
where $m_i, n_i \in \Z$ for $i = 1, 2$,  and
\item [(b)] $  k_1 m_1 n_1  + k_2 m_2 n_2 = \ell \pmod  2.$
\end{enumerate}
Therefore,  
every point $[z;u] \in M_\pm$ that satisfies (a) and (b) is a singular point of the orbifold with isotropy group $\Z_2$; otherwise, $M_\pm$ is smooth.
In contrast, the stabilizer group $\{ \lambda \in \C^\times \mid \lambda \cdot [z;u] = [z;u]\}$ is $\Z_2$  if (a) holds but (b) does not; otherwise,
the action of $\C^\times$ on $M_\pm$ is free.
\end{Remark}

\section{Hamiltonian actions with fixed points}
\labell{ss:fixed}

In this final section, we consider Hamiltonian circle actions
on symplectic manifolds {\em with} fixed points.
However, each regular  reduced space 
is still  either symplectomorphic to
a tame K3 surface or diffeomorphic to the Kummer surface $T/\Z_2$.
As in the previous section, we don't prove general theorems but simply 
use ideas from \cite{TWa} to construct the
examples we need.
More precisely, we prove the following:

\begin{Proposition}\labell{existsfixed}
There exist symplectic
manifolds $(\Tilde{M}_\pm,\Tilde{\omega}_\pm)$ 
with circle actions and
proper moment maps $\Tilde\Psi_\pm \colon \Tilde{M}_\pm \to \R$,
and $\epsilon > 0$ so that:
\begin{enumerate}
\item
The preimages $\Tilde{\Psi}_\pm\inv(\pm (-\infty, 1 + \epsilon))$ each contain exactly $16$ fixed
points; each lies in $\Tilde\Psi_\pm(\pm 1)$ 
and has weights $\pm \{-2,1,1\}$.
\item 
Let $(M_\pm, \omega_\pm,\Psi_\pm)$ be the symplectic manifolds with locally free circle actions and proper moment maps described in Propositions~\ref{existslf}.
There exist 
$a_\pm < b_\pm$ in $\pm (0,1)$ so that
$\Psi_\pm\inv(a_\pm,b_\pm)$ and $\Tilde{\Psi}_\pm\inv(a_\pm,b_\pm)$ 
are equivariantly symplectomorphic.
\item 
For all $t \in \pm(1,1+\epsilon)$,
the reduced space 
$\Tilde{M}_\pm \mod_{\!  t} \, S^1$  is symplectomorphic to a tame K3 surface 
$(\Tilde{X},\Tilde{I},(\Tilde{\sigma}_\pm)_t)$; moreover,

\begin{itemize}
\item $\big( (\Tilde{\sigma}_\pm)_t, (\Tilde{\sigma}_\pm)_t \big)  = -4 \pm 16 t - 4 t^2$; and
\item $\big[(\Tilde{\sigma}_\pm)_t\big] = \Tilde{\kappa} -  t \Tilde{\eta}_\pm,$ 
where $\Tilde{\kappa},\Tilde{\eta}_\pm$ induce a primitive embedding $\Z^2 \hookrightarrow H^2(\Tilde{X};\Z)$.
\end{itemize}
\end{enumerate}
\end{Proposition}

To construct these examples, we need to introduce some notation:
Given a holomorphic $\C^\times$ action on a complex
manifold $(M,J)$, let $\xi_M$ be the vector field associated
to the induced $S^1 \subset \C^\times$ action, and let $\Omega^2(M)^{S^1}$
denote the set of $S^1$ invariant two-forms on $M$.

We also need the following three results from \cite{TWa}, which
is joint with J. Watts.  More specifically, they correspond
to \cite[Proposition 3.1]{TWa},  \cite[Proposition 7.1]{TWa}, and
\cite[Proposition 7.8]{TWa}, respectively.

\begin{Proposition}\labell{reduction}
Let $(M,J)$ be a complex manifold with a holomorphic $\C^\times$-action, 
a symplectic form $\omega \in \Omega^2(M)^{S^1}$ satisfying
$\omega(\xi_M,J(\xi_M)) > 0$ on $M \smallsetminus M^{S^1}$,
and a
moment map $\Psi \colon M \to R$.
If $a \in \R$ is a regular value of $\Psi$ and
$U_a := \C^\times \cdot \Psi\inv(a)$, then
the following hold:
\begin{itemize}
\item The quotient $U_a/C^\times$ is a naturally a complex orbifold
\item 
There is a complex structure $J_a$ on the reduced space
$M \mod_{\! a}\, S^1$   
so that that  the inclusion $\Psi\inv(a) \hookrightarrow U_a$ induces
a biholomorphism $M \mod_{\!a } \, S^1 \to U_a/\C^\times$.
\item  If  $J$ tames $\omega$  near $\Psi\inv(a)$,
then $J_a$ tames the reduced symplectic form on $M \mod_{\! a} \, S^1$. 
\end{itemize}
\end{Proposition}

\begin{Proposition} \labell{TWa}
Fix $a \in \R$.
Let  $(M,J)$ be a complex manifold with a holomorphic $\C^\times$ action, a
symplectic form $\omega \in \Omega^2(M)^{S^1}$
that is tamed near $\Psi\inv(a)$ and satisfies
$\omega(\xi_M,J(\xi_M)) > 0$ on $M \smallsetminus M^{S^1}$,
and  a proper moment map  $\Psi \colon M \to R$.
Assume that the  $S^1$ action on  $\Psi\inv(a)$ is free except
for $k$ orbits with stabilizer $\Z_2$.
Then for sufficiently small $\epsilon > 0$ there exists a complex manifold  $(\Tilde{M},\Tilde{J})$ with 
a holomorphic $\C^\times$ action,   a
symplectic form $\Tilde{\omega} \in \Omega^2(\Tilde{M})^{S^1}$ 
satisfying  $\Tilde{\omega}( \xi_{\Tilde{M}}, \Tilde{J}(\xi_{\Tilde{M}})) > 0$ on $\Tilde{M} \smallsetminus \Tilde{M}^{S^1}$, 
and a proper moment map
$\Tilde{\Psi} \colon \Tilde{M} \to  \R$
so that the following hold:
\begin{enumerate}
\item  $\Tilde\Psi\inv(a-\epsilon,a]$ 
contains exactly $k$ fixed points; each lies in $\Tilde\Psi\inv(a)$
and has   weights $\{-2,1,\dots,1\}$.
\item There is an equivariant symplectomorphism  
$\Tilde\Psi\inv(-\infty,a -\epsilon/2) \to \Psi\inv(-\infty,a -\epsilon/2)$ that
induces a biholomorphism 
$\Tilde M \mod_{\! t} \, S^1 \to M \mod_{\! t}\, S^1$
for all regular $t  \in (-\infty, a - \epsilon/2)$.
\item $\Tilde\omega$ is tamed on $\Tilde\Psi\inv(a-\epsilon, a+\epsilon)$.
\end{enumerate}
\end{Proposition}

\begin{Proposition}\labell{simple fixed point}
Let $(M,J)$ be a complex manifold with a holomorphic
$\C^\times$-action, a symplectic form $\omega \in \Omega^2(M)^{S^1}$ satisfying $\omega(\xi_M, J(\xi_M)) > 0$ on $M \ssminus M^{S^1}$, and a
proper moment map $\Psi \colon M \to R$.
Assume that $\dim_\C M > 1$,
and that $\Psi\inv(a,b) \cap M^{S^1}$ contains exactly  $k$ fixed points;
each lies in $\Psi\inv(0)$  and has 
weights $\{-2,1,\dots,1\}$.
Then there exists a complex orbifold $(X,I)$, 
and classes $\kappa,\eta \in H^2(X;\R)$, 
such that:
\begin{enumerate}
\item For all $t \in (a,0)$, the reduced space $M \mod_{\! t}\, S^1$ is biholomorphically symplectomorphic to $(X,I,\sigma_t)$, 
where  $\sigma_t \in \Omega^2(X)$ satisfies
$[\sigma_t] = \kappa - t \eta \in H^2(X;\R).$
\item For all $t \in (0,b)$, the reduced space 
$M \mod_{\! t}\, S^1$ is biholomorphically symplectomorphic to 
$(\Hat{X},\Hat{I},\Hat{\sigma}_t)$, where  
$\Hat{\sigma}_t \in \Omega^2(\Hat X)$ satisfies
\begin{equation}\labell{sfp2}
[\Hat{\sigma}_t] = 
q^*\kappa - t\, q^*\eta - t/2 \sum_{i=1}^k \cE_i \in H^2(\Hat{X};\R).
\end{equation}
\end{enumerate}
Here, $(\Hat{X},\Hat{I})$ is the blow-up of $(X,I)$ at isolated $\Z_2$ singularities $p_1,\dots,p_k$,  
the map $q \colon \Hat{X} \to X$ is the blow-down, 
and $\cE_i$ is the Poincare dual of the exceptional divisor $q^{-1}(p_i)$ for all $i$.
\end{Proposition}

\begin{Remark}\labell{reverse}
Alternatively, by \cite[Remark 4.4]{TWa}, Proposition~\ref{TWa} 
is true with the following  
modifications:
First,  in claim (1), the preimage  is $\Tilde\Psi\inv[a,a+\eps)$ instead of $\Tilde\Psi\inv(a-\eps,a]$, and the weights are $\{2,-1,\dots,-1\}$ instead of 
$\{-2,1,\dots,1\}$.
Second, in claim (2), the set $(-\infty,a-\epsilon/2)$ is replaced
by
$(a + \epsilon/2, \infty)$.

Similarly,  Proposition~\ref{simple fixed point} is true with the following 
modifications:
Assume that  the weights 
are $\{2,-1,\dots,-1\}$ 
at each fixed points in $\Psi\inv(0)$,
instead of 
$\{-2,1,\dots,1\}$.
In this case, claim (1) holds for all $t \in (0,b)$, not all $t \in (a,0)$.
Claim (2) holds for all $t \in (a,0)$, not all $t \in (0,b)$;
moreover,  \eqref{sfp2} is replaced by
$$[\Hat{\sigma}_t] = 
q^*\kappa - t\, q^*\eta + t/2 \sum_{i=1}^k \cE_i \in H^2(\Hat{X};\R).$$

\end{Remark}

Let  $\pi \colon T \to T/\Z_2$ be the quotient map.
Let $(\Hat{X},\Hat{I})$ be the K3 surface formed by
blowing up the quotient $T/\Z_2$ at the $16$ isolated $\Z_2$ singularities $p_1,\dots,p_{16}$.
Finally, let  $q \colon \Hat{X} \to T/\Z^2$
be the blow-down map, 
and  let $\cE_j \in H^2(\Hat{X};\Z)$ be  Poincare dual to the 
exceptional divisor $q^{-1}(p_i)$ for 
all $i$.
We will use the criteria below to prove claim (3) of Proposition~\ref{existsfixed}

\begin{Lemma}\labell{primitive}
Given $y \in  H^2(T/\Z_2;\R)$  such that 
$ \pi^*(y/2)$ is a primitive class in $H^2(T;\Z)$, 
the pull-back $q^*(y)$ is a primitive class in $H^2(\Hat{X};\Z)$.
If additionally  $x \in H^2(\Hat{X};\Z)$  satisfies $(\cE_i,x) = \pm 1$ for some $i \in \{1,\dots,16\},$
then $q^*(y), x$ induce a primitive embedding $\Z^2 \hookrightarrow H^2(\Hat{X};\Z)$.
\end{Lemma}

\begin{proof}
Let $\Hat{T}$ be the blow-up of $T$ at the 16 fixed points of the $\Z_2$ action,
and let $\sigma \colon \Hat{T} \to T$ be the blow-down map.
The involution on $T$ induces an involution on $\Hat{T}$ that fixes the exceptional divisors. 
The quotient $\Hat{T}/\Z_2$
is naturally isomorphic to $\Hat{X}$;  let $\Hat{\pi} \colon \Hat{T} \to \Hat{X}$ be the quotient map.
Note that
$q \circ \Hat{\pi} = \pi \circ \sigma$. 

Define $\alpha \colon H^*(T;\Z) \to H^*(\Hat{X};\Z)$ by $\alpha = \Hat{\pi}_{!} \circ \sigma^*$,
where $\Hat{\pi}_{!} \colon H^*(\Hat{T};\Z)  \to H^*(\Hat{X};\Z)$ is the push-forward map.
By \cite[Corollary VIII.5.6]{BPV}, the lattice $\alpha(H^2(\Hat{X};\Z))$ is  primitive.
Moreover, 
$\Hat{\pi}_{!}(\Hat\pi^*(z)) = 2 z$
for all $z \in H^*(\Hat{X};\Z)$
because  $\Hat\pi$ has degree $2$.
Therefore,
given $y \in H^2(T/\Z_2;\R)$ such that $\pi^*(y/2)$ is a primitive class in $H^2(T;\Z)$, 
its image
$$\alpha \left(\pi^*(y/2)\right) =   \Hat{\pi}_{!}( \sigma^*(\pi^*(y/2)))
=   \Hat{\pi}_{!}( \Hat{\pi}^*(q^*(y)))/2 =  q^*(y).$$
is a primitive class in $H^2(\Hat{X};\Z)$.
Since $(\cE_j,q^*(y)) = 0$ for all $j$, the second claim follows immediately.
\end{proof}

We are now ready to begin our proof.
\begin{proof}[Proof of Proposition~\ref{existsfixed}]

Let $(M_+,J_+)$ be the complex manifold
with  locally free holomorphic $\C^\times$
action, 
symplectic form 
$\omega_+ \in \Omega^2(M_+)^{S^1}$ satisfying
$\omega_+(\xi_{M_+},J_+(\xi_{M_+})) > 0$, 
proper moment map  $\Psi_+ \colon
M_+ \to \R$, and $\C^\times$ invariant map $\pi_+ \colon M_+ \to T/\Z_2$ 
described in Proposition~\ref{existslf}.
For all $t \in \R$,  the map  $\pi_+$ induces a symplectomorphism from the reduced space
$M_+ \mod_{\! t}\, S^1$ to the Kummer surface $T/\Z_2$ with
$(\sigma_+)_t \in \Omega^2(T/\Z_2)$, where
\begin{equation}\labell{ef1}
(\sigma_+)_t = dz_1 dz_2 + d \zbar_1 d \zbar_2 + \sqrt{-1}\,  t  dz_1 d \zbar_1 +  \sqrt{-1}\, t dz_2 d \zbar_2 
.\end{equation}
In particular, the $S^1$ action on $\Psi_+\inv(t)$ is free except for $16$ orbits
with  stabilizer $\Z_2$.
Moreover,
$\pi_+$ induces a biholomorphism from $M_+/\C^\times$ to $T/\Z_2$,
and so $U_t := \C^\times \cdot \Psi_+\inv(t) = M_+$ for all $t \in \R$.
Hence, Proposition~\ref{reduction} implies that 
there is a natural complex structure on the reduced space
$M_+ \mod_{\!t }\, S^1$, and  
the symplectomorphism from $M_+ \mod_{\! t}\, S^1$ to $T/\Z_2$ described above is biholomorphism.
Finally, $J_+$ tames $\omega_+$ on $\Psi_+\inv(0,\infty)$.

Therefore, 
we can apply Proposition~\ref{TWa} with $a = 1$
to  construct
a complex manifold $(\Tilde{M}_+, \Tilde{J}_+)$  with a holomorphic
$\C^\times$ action,
a symplectic form
$\Tilde{\omega}_+ \in \Omega^2(\Tilde{M}_+)^{S^1}$  
satisfying $\Tilde{\omega}_+ (\xi_{\Tilde{M}_+}, \Hat{J}_+(\xi_{\Tilde{M}_+})) > 0$ 
on $\Hat{M}_+ \smallsetminus \Hat{M}_+^{S^1}$,
and a  proper moment map
$\Hat{\Psi}_+ \colon \Hat{M}_+ \to \R$.
Since we may assume that $\epsilon < 1$,
claim (2) implies that their exists a non-empty open set $U_+ \subset (0,1)$ 
and an equivariant symplectomorphism
from $\Tilde\Psi_+\inv(U_+)$ to $\Psi_+\inv(U_+)$ that induces
a biholomorphism from $\Tilde{M}_+ \mod_{\! t}\, S^1$  to $ M_+ \mod_{\! t}\, S^1$
for all $t \in U_+$.
In particular, the reduced space $\Tilde{M}_+ \mod_{\! t}\, S^1$
is biholomorphically symplectomorphic to the Kummer surface
$T/\Z_2$ with the symplectic form
$(\sigma_+)_t \in \Omega^2(T/\Z_2)$ for all $t \in U_+$.
Since $M_+$ has no fixed points, claims (1) and (2) together imply that,
after possibly shrinking $\epsilon$,
the preimage  $\Hat\Psi_+\inv(-\infty,\epsilon)$  contains exactly 
$16$ fixed points; each lies in $\Hat\Psi_+\inv(1)$ 
and has weights $\{-2,1,1\}$.  
Finally, by claim (3), the symplectic form 
$\Tilde\omega_+$ is tamed on $\Tilde\Psi_+\inv(1-\epsilon,1+\epsilon)$.

Hence, by Proposition~\ref{simple fixed point},
for all $t \in (1, 1+ \epsilon)$,
the reduced space $\Tilde{M}_+ \mod_{\! t} \, S^1$
is biholomorphically symplectomorphic  to the K3 surface
$(\Hat{X},\Hat{I})$ with symplectic form $(\Hat{\sigma}_+)_t \in \Omega^2(T/\Z_2)$,
where
\begin{equation}\labell{ef2}
\bigl[(\Hat\sigma_+)_t\bigr] =  q^*\bigl(\bigl[(\sigma_+)_t\bigr]\bigr) - (t - 1)/2 \sum_{i=1}^{16} \cE_i.
\end{equation}
Moreover, by  Proposition~\ref{reduction}, the reduced complex structure $\Hat{I}$ tames
$(\Hat{\sigma}_+)_t$ for all $t \in (1,1+\epsilon)$.
Combining \eqref{ef1} and \eqref{ef2}, 
we see that 
$[(\Hat{\sigma}_+)_t] = \Hat{\kappa} - t \Hat{\eta}_+$ 
for all $t \in (1,1+\epsilon)$, 
where
\begin{gather*}
\Hat{\kappa} = q^*\big([dz_1 dz_2 + d \zbar_1 d \zbar_2]\big) +  1/2 \sum_{i=1}^{16} \cE_i,  \quad \mbox{and} \\   
\Hat{\eta}_+ =   - q^*\big([  \sqrt{-1} \, dz_1 d \zbar_1 + \sqrt{-1} \, dz_2 d \zbar_2]\big) + 1/2 \sum_{i=1}^{16} \cE_i.
\end{gather*}

A direct calculation shows that
\begin{multline*}
\int_{T/Z_2} (dz_1 dz_2 + d \zbar_1 d \zbar_2) \wedge (dz_1 dz_2 + d \zbar_1 d \zbar_2) = \\
\frac{1}{2} \int_T (dz_1 dz_2 + d \zbar_1 d \zbar_2) \wedge (dz_1 dz_2 + d \zbar_1 d \zbar_2) = 4.\end{multline*}
Therefore, the cup product of  $q^*\big([dz_1 dz_2 + d \zbar_1 d \zbar_2]\big) 
\in H^2\big(\Hat{X}\big)$ with itself is $4$.
Since 
$$( q^*\big( [dz_1 dz_2 + d \zbar_1 d \zbar_2] \big), \cE_i) = 0
\mbox{ and } (\cE_i,\cE_j) = -2 \delta_{ij}$$ for all $i$ and $j$, this implies that
\begin{equation*}\labell{prod1}
(\Hat{\kappa}, \Hat{\kappa}) = 4 +  16 (-2)/2^2  = - 4.
\end{equation*}
By similar arguments,
\begin{equation*}\labell{prod2}
(\Hat{\eta}_+, \Hat{\eta}_+) =  - 4
\ \  \mbox{and} \   \ 
(\Hat{\eta}_+, \Hat{\kappa}) =  -8. 
\end{equation*}
Thus,
 $(\Hat{\kappa}  - \Hat{\eta}_+, \Hat{\kappa} - t \Hat{\eta}_+)  = -4 t^2+ 
16 t - 4 $,
as required.

Finally, since $$1/2 \left[ dz_1 dz_2 + d \zbar_1 d \zbar_2 + 
\sqrt{-1}\,  dz_1 d \zbar_1 + \sqrt{-1}\, dz_2 d \zbar_2\right]$$
is a primitive integral  in $H^2(T;\Z)$, Lemma~\ref{primitive} implies that 
that $\Hat{\kappa} - \Hat{\eta}_+$  is a primitive  class  in $H^2(\Hat{X};\Z)$.
Moreover, since   $\Hat{\eta}_+$ is  Euler class
of the circle bundle  $\Tilde{\Psi}_+\inv(t) \to \Tilde{M}_+ \mod_{\! t}\, S^1
\simeq \Hat{X}$
for all $t \in (1, 1+ \epsilon)$,  the class 
$\Hat{\eta}_+$ lies in the integral cohomology 
$H^2(\Hat{X};\Z)$.
Hence,  since
the cup product $(\Hat{\eta}_+,\cE_i) = -1$ for all $i$, 
Lemma~\ref{primitive} implies that
$\Hat{\kappa}, \Hat{\eta}_+$ induce a primitive embedding, as required.

The construction of $\Tilde{M}_-$ is analogous, except that in this case
we begin with the example $M_-$ given in Proposition~\ref{existslf},
instead of taking $M_+$.
We then  apply  the variants of Proposition~\ref{TWa} and ~\ref{simple fixed point} 
suggested in  Remark~\ref{reverse}, and in the former case we take $a = -1$, instead of $a = 1$.
In this case,
for all $t \in (-1-\epsilon, -1)$,
the reduced space $\Tilde{M}_- \mod_{\! t} \, S^1$
is biholomorphically symplectomorphic to the tame K3 surface
$(\Hat{X},\Hat{I}, (\Hat{\sigma}_-)_t)$ with
$[(\Hat{\sigma}_-)_t] = \Hat{\kappa} - t \Hat{\eta}_-$, where
$$
\Hat{\eta}_- =   - q^*\big([ \sqrt{-1} \, dz_1 d \zbar_1 + \sqrt{-1}\, dz_2 d \zbar_2]\big) - 1/2 \sum_{i=1}^{16} \cE_i.
$$

\end{proof}


\begin{thebibliography}{000}


\bibitem[BPV]{BPV} W. Barth, C. Peters, and A. Ven de Ven, {\em Compact Complex Surfaces},
Springer-Verlag, Berlin Heidleberg New York Tokyo, 1984.

\comment{
\bibitem[CHS1]{CHS1} Y. Cho, Y. Hwang, and D. Suh,
{\em Semifree Hamiltonian circle actions on $6$-dimensional symplectic manifolds with non-isolated fixed point set}, arXiv:1005.0193
}

\bibitem[Do]{Do} S.K. Donaldson, {\em Polynomial invariant for smooth four-manifolds}, Topology \textbf{29} (1990), 257-315.

\bibitem[DH]{DH} J.J Duistermaat and G.J Heckman, {\em On the variation in the cohomology of the symplectic form on the reduced phase space}, Invent. Math. {\bf 69} (1982), 259-268.

\comment{
\bibitem[Fa]{Fa} A. Fanoe, {\em Hamiltonian $S^1$ actions with isolated fixed points on $6$-dimensional symplectic manifolds}.
}

\bibitem[Fe]{Fe} K.E. Feldman, {\em Hirzeburch genus of a manifold supporting a Hamiltonian circle action}, Russian Math. Surveys {\bf 56} (2001), 978-979.

\bibitem[Fr]{Fr}  T. Frankel, {\em Fixed points and torsion on K\"ahler manifolds}, 
Ann. of  Math.  {\bf 70 } (1959), 
1-8.

\bibitem[Gia]{Gia} A. Giacobbe, {\em Convexity of multi-valued momentum maps}, Geom. Dedicata {\bf 111} (2005), 1-22.

\bibitem[Gin]{Gin} V. Ginzburg, {\em Some remarks on symplectic actions of compact groups}, Math. Z. {\bf 210} (1992),  625-640.

\bibitem[Go05]{Go05} L. Godinho, {\em On certain symplectic circle actions}, J. Symplectic Geom. {\bf  3} (2005), 357-383.

\bibitem[Go06]{Go06} L. Godinho, {\em Semifree symplectic circle actions on $4$-orbifolds}, Trans. Amer. Math. Soc. {\bf 358} (2006), 4919-4933.

\comment{
\bibitem[Gr]{Gr} W. Graham, {\em Logarithmic convexity of push-forward measures},
Invent.\ Math.\ 123, 315-322 (1996).
}


\bibitem[Ja]{Ja} D. Jang,
{\em Symplectic periodic flows with exactly three equilibrium points},
Ergodic Theory Dynam. Systems {\bf 34}  (2014), 1930-1963.

\comment{
\bibitem[J2]{J2} D. Jang, 
{\em Symplectic circle actions with isolated fixed points}, arXiv:1412.4169.
}

\comment{
\bibitem[Ka]{Ka} B. Kanshige, {\em On semifree symplectic circle actions}, thesis.
}

\comment{
\bibitem[KRT]{KRT} J. Kendra, Y. Rudnyak, and A. Tralle, {\em Symplectically aspherical manifolds},  J. Fixed Point Theory Appl. {\bf 3} (2008) 1-21.
}

\bibitem[Ki06]{Ki06} M. Kyu Kim, {\em Frankel's theorem the symplectic category}, Trans. Amer. Math. Soc. {\bf 358} (2006), 4367-4377.

\comment{
\bibitem[Ki07]{Ki07} M. Kyu Kim, {\em Topological method does not work for Frankel-McDuff conjecture}, Journal of Chungcheong mathematical society, {\bf 20}
(2007),  31-35.
}


\bibitem[Ko]{Ko} D. Kotschick,
 {\em Free circle actions with contractible orbits on symplectic manifolds}, 
Math. Z. {\bf 252} (2006),  19-25.

\bibitem[Li]{Li} P. Li {\em The rigidity of Dolbeault-type operators and symplectic circle actions}, 
Proc. Amer. Math. Soc. {\bf 140} (2012) 1987-1995.

\comment{
\bibitem[Lin]{Lin} Y. Lin {\em The log-concavity conjecture for the Duistermaat-Heckman measure revisited}, Int. Math. Res. Not. IMRN (2008) Vol. 2008 article ID rnn027, 19 pages.
}

\comment{
\bibitem[LP]{LP} Y. Lin, A. Pelayo, {\em Log-Concavity and symplectic flow},
arXiv:1207.1335, 32 pages.
}

\bibitem[LO]{LO} G. Lupton and J. Oprea, {\em Cohomologically symplectic spaces: toral actions and the Gottlieb group}, Trans. Amer. Math. Soc. {\bf 347} (1995), 261-288.


\bibitem[MPR]{MPR} R. Mazzeo, A. Pelayo, and T. Ratiu, {\em $L^2$-cohomology and complete Hamiltonian manifolds}, J.  Geom. and Phys. 
{\bf 87} (2015),  305-313. 

\bibitem[Mc]{Mc} D. McDuff,  {\em The moment map for circle
actions on symplectic manifolds.}  J.\  Geom.  Phys.
{\bf  5} (1988), 149--160.

\bibitem[MS]{MS} D. McDuff and D. Salamon,  {\em Introduction to symplectic topology}, Oxford University Pres, Oxford, 1998.

\bibitem[On]{On} K. Ono, {\em Some remarks on group actions in symplectic geometry},
J. Fac. Sci. Univ. Tokyo Sect. IA Math. {\bf 35} (1988),  431-437.

\comment{
\bibitem[PR]{PR} A. Pelayo and T. Ratiu, Applying Hodge theory to detect Hamiltonian flows, arXiv:1005.2163
}

\bibitem[PT]{PT} A. Pelayo and S. Tolman, {\em Fixed points of symplectic periodic flows}, Ergodic Theory Dynam.\ Systems {\bf 31} (2011),  1237-1247.

\comment{
\bibitem[Pe]{Pe} A. Pelayo, Symplectic actions of non-Hamiltonian type. arXim:1501.06480.
}


\bibitem[Ro]{Ro} F. Rochon, {\em Rigidity of Hamiltonian actions}, Canad. Math. Bull. {\bf 46} (2006) 277-290.

\comment{
\bibitem[Sl]{Sl} On Generalized Moment maps for symplectic compact group actions, 
arXiv:math/0304487
}

\bibitem[TWe]{TWe} S. Tolman and J. Weitsman, {\em Semifree symplectic circle action with isolated fixed points}, Topology {\bf 39} (2000),   299-309. 

\bibitem[TWa]{TWa} S. Tolman and J. Watts, {\em Tame circle actions}, 
preprint (2015),  arXiv:1510.01721.

\end{thebibliography}
\end{document}